\documentclass[11pt]{article}
\usepackage{amscd, amssymb, amsmath, calc, enumerate}
\begin{document} 
\renewcommand{\thesubsection}{\arabic{subsection}}
\newenvironment{eq}{\begin{equation}}{\end{equation}}
\newenvironment{proof}{{\bf Proof}:}{\vskip 5mm }
\newenvironment{rem}{{\bf Remark}:}{\vskip 5mm }
\newenvironment{remarks}{{\bf Remarks}:\begin{enumerate}}{\end{enumerate}}
\newenvironment{examples}{{\bf Examples}:\begin{enumerate}}{\end{enumerate}}  
\newtheorem{proposition}{Proposition}[subsection]
\newtheorem{lemma}[proposition]{Lemma}
\newtheorem{definition}[proposition]{Definition}
\newtheorem{theorem}[proposition]{Theorem}
\newtheorem{cor}[proposition]{Corollary}
\newtheorem{conjecture}{Conjecture}
\newtheorem{pretheorem}[proposition]{Pretheorem}
\newtheorem{hypothesis}[proposition]{Hypothesis}
\newtheorem{example}[proposition]{Example}
\newtheorem{remark}[proposition]{Remark}
\newtheorem{ex}[proposition]{Exercise}
\newtheorem{cond}[proposition]{Conditions}
\newtheorem{cons}[proposition]{Construction}
\newcommand{\llabel}[1]{\label{#1}}
\newcommand{\comment}[1]{}
\newcommand{\sr}{\rightarrow}
\newcommand{\lr}{\longrightarrow}
\newcommand{\xr}{\xrightarrow}
\newcommand{\dw}{\downarrow}
\newcommand{\bdl}{\bar{\Delta}}
\newcommand{\zz}{{\bf Z\rm}}
\newcommand{\zq}{{\bf Z}_{qfh}}
\newcommand{\nn}{{\bf N\rm}}
\newcommand{\qq}{{\bf Q\rm}}
\newcommand{\nq}{{\bf N}_{qfh}}
\newcommand{\oo}{\otimes}
\newcommand{\uu}{\underline}
\newcommand{\ih}{\uu{Hom}}
\newcommand{\af}{{\bf A}^1}
\newcommand{\wt}{\widetilde}
\newcommand{\gm}{{\bf G}_m}
\newcommand{\dsr}{\stackrel{\sr}{\scriptstyle\sr}}
\newcommand{\PP}{$P_{\infty}$}
\newcommand{\tp}{\tilde{D}}
\newcommand{\HH}{$H_{\infty}$}
\newcommand{\ii}{\stackrel{\scriptstyle\sim}{\sr}}
\newcommand{\BB}{_{\bullet}}
\newcommand{\D}{\Delta}
\newcommand{\colim}{{\rm co}\hspace{-1mm}\lim}
\newcommand{\cf}{{\it cf} }
\newcommand{\msf}{\mathsf }
\newcommand{\mcal}{\mathcal }
\newcommand{\ep}{\epsilon}
\newcommand{\tl}{\widetilde}
\newcommand{\ub}{\mbox{\rotatebox{90}{$\in$}}}
\newcommand{\piece}{\vskip 3mm\noindent\refstepcounter{proposition}{\bf
\theproposition}\hspace{2mm}}
\newcommand{\subpiece}{\vskip 3mm\noindent\refstepcounter{equation}{\bf\theequation}\hspace{2mm}}{\vskip
3mm}
\newcommand{\maintwo}{\cite[Theorem 3.74]{Red2sub}}
\newcommand{\anotherth}{\cite[Theorem 4.24]{Red2sub}}
\newcommand{\refsltensor}{\cite[Lemma 5.15]{oversub}}
\newcommand{\refdualten}{\cite[Lemma 5.17]{oversub}}
\newcommand{\refunext}{\cite[Lemma 5.18]{oversub}}
\newcommand{\refdt}{\cite[Prop. 5.20]{oversub}}
\newcommand{\refdonotneed}{\cite[Lemma 6.9]{oversub}}
\newcommand{\refrl}{\cite[Lemma 6.11]{oversub}}
\newcommand{\refintres}{\cite[Lemma 6.12]{oversub}}
\newcommand{\refmotdual}{\cite[Lemma 6.14]{oversub}}
\newcommand{\refspcase}{\cite[Lemma 6.15]{oversub}}
\newcommand{\refwheneq}{\cite[Lemma 6.23]{oversub}}
\newcommand{\refappmain}{\cite[Th. 8.3]{oversub}}

\numberwithin{equation}{subsection}
%
%
\begin{center}
{\Large\bf On motivic cohomology with $\zz/l$-coefficients\footnote{For the referee: The theorem numbers in references to \cite{Red2sub} and \cite{oversub} are given relative to the versions which are enclosed to this submission.}}\\
\vskip 4mm
{\large Vladimir Voevodsky}\\
{\em December 2008}
\end{center}
\vskip 4mm
\begin{abstract}
In this paper we prove the conjecture of Bloch and Kato which relates Milnor's K-theory of a field with its Galois cohomology as well as the related comparisons results for motivic cohomology with finite coefficients in the Nisnevich and etale topologies. \end{abstract}
\tableofcontents
\subsection{Introduction}
\llabel{s0}
In this paper we prove the Bloch-Kato conjecture relating the Milnor
K-theory and etale cohomology. It is a continuation of \cite{MCpub}
where the particular case of $\zz/2$-coefficients (``Milnor's
conjecture'') was established and we refer to the introduction to
\cite{MCpub} for general discussion about the Bloch-Kato conjecture. 

The goal of Sections \ref{s1}, \ref{s2} is to prove Theorem
\ref{maincomp} which relates two types of cohomological operations in
motivic cohomology. One of the operations appearing in the theorem is
defined in terms of symmetric power functors in the categories of
relative Tate motives and another one in terms of the motivic reduced
power operations introduced in \cite{Redpub}. Our proof of this theorem is inspired by
\cite{kraines} and uses a uniqueness argument based on the computations of \cite{Red2sub}. This is the only place where the results of \cite{Red2sub} (and therefore of \cite{SRFsub}) are used and the only place where the results of \cite{oversub} are used in an essential way.

In Section \ref{s3} we consider motives over a special class of
simplicial schemes which are called ``embedded simplicial schemes''
(see \cite{oversub}). Up to an equivalence, embedded simplicial schemes
correspond to subsheaves of the constant one point sheaf on $Sm/k$
i.e. with classes of smooth varieties such that
\begin{enumerate} 
\item if $X$ is in the class and $Hom(Y,X)\ne \emptyset$ then $Y$ is
in the class, and
\item if $U\sr X$ is a Nisnevich covering and $U$ is in the class then
$X$ is in the class.
\end{enumerate}
In particular for a symbol $\uu{a}=(a_1,\dots,a_n)$ we have an
embedded simplicial scheme ${\cal X}_{\uu{a}}$ associated with the
class of all splitting varieties for $\uu{a}$ and the motivic cohomology
of ${\cal X}_{\uu{a}}$ plays a key role in our proof of the Bloch-Kato
conjecture.

The main goal of Section \ref{s3} is to prove a technical result - Theorem
\ref{degree}, which is used in the next section to establish the
purity of the generalized Rost motives. We call this result ``a
motivic degree theorem'' because it is analogous to the simplest
degree formula for varieties which asserts that a morphism from a
$\nu_n$-variety to a variety without zero cycles of degree
prime to $l$ has degree prime to $l$. The main difference between the
standard degree formula and our result is that the target of the
morphism in our case is a motive rather than a variety. As a
consequence of this higer generality we also require stronger
conditions on the target than simply the absence of zero cycles of
degree prime to $l$.

In Section \ref{s4} we introduce the construction which represents
the key difference between the case of $\zz/2$-coefficients and
$\zz/l$-coefficients for $l>2$. In the $\zz/2$-coefficients case the
Pfister quadrics provide canonical $\nu_{n-1}$-splitting varieties for
symbols of length $n$. The explicit nature of these varieties made it
possible for Markus Rost to invent a simple geometric argument
which showed that the motive of a Pfister quadric splits as a direct
sum of an ``essential part'' (which we called the Rost motive in
\cite{MCpub}) and a ``non essential part'' which can be ignored as far
as our goals are concerned. The fact that the Rost motive is a direct
summand of the motive of a $\nu_{n-1}$-variety and at the same
time has a description in terms of Tate motives over the embedded
simplicial scheme ${\cal X}_{\uu{a}}$ defined by the symbol puts
strong restrictions on the motivic cohomology of ${\cal
X}_{\uu{a}}$. These restrictions allowed us to
reformulate the vanishing result needed for the proof of the Milnor
conjecture in terms of a motivic homology group of the Pfister
quadric which can be analyzed geometrically. 

A direct extention of these arguments to the $l>2$ case fails for two
main reasons. On the one hand we do not have nice geometric models for
$\nu_{n-1}$-splitting varieties for symbols of length $n$. On the
other hand the argument which for $l=2$ transfers the vanishing problem to a motivic homolopgy group having an explicit 
geometric description fails to produce the same result for $l>2$
ending in a group which is not any easier to understand than the
original one. 

We show in Section \ref{s4} that any embedded simplicial scheme
$\cal X$ which has a non-trivial motivic cohomology class of certain
bidegree and such that the corresponding class of varieties contains a
$\nu_n$-variety defines a {\em generalized Rost motive}. This motive
is constructed from the Tate motives over $\cal X$ and we use the
motivic degree theorem of the previous section to prove that it is a
direct summand of the motive of any $\nu_{n}$-variety over
$\cal X$. The key ingredient of the proof is the relation between the
$(l-1)$-st symmetric powers and Milnor operations $Q_i$ provided by
Theorem \ref{maincomp} and Lemma \ref{scomp}. 

Generalized Rost motives unify two previously known families of
motives - the Rost motives for $l=2$ discussed above and the motives
of cyclic field extensions of prime degree. The generalized Rost
motives correspond to motivic cohomology classes which have
$\nu_{n}$-splitting varieties in the same way as the motives of the
cyclic field extensions correspond to the motivic cohomology classes
in $H^{1,1}(k,\zz/l)$. 

In Section \ref{s5} we give a proof of the Bloch-Kato conjecture
based on the results of the previous sections, \cite{MCpub} and Theorem
\ref{Rost2}. 

The approach to the Bloch-Kato conjecture used in the present paper goes back to the fall of 1996. The proof of Theorem \ref{maincomp} in the first version of this paper (see \cite{motcoh}) was based on a lemma (\cite[Lemma 2.2]{motcoh}) the validity of which is, at the moment, under serious doubt. In  \cite{patching}, C. Weibel suggested another approach to the proof of \ref{maincomp}. In the present version of the paper we use a modified version of Weibel's approach in which \cite[Lemma 2.2]{motcoh} is replaced by Lemma \ref{newlemma}.

I would like to specially thank several people who helped me to
understand things used in this paper. Pierre Deligne for explaining to
me how to define sheaves over simplicial schemes and for help with the
computation of $H^{*}(B{\bf G}_a, {\bf G}_a)$. Peter May for
general remarks on tensor triangulated categories. Fabien Morel for
helping me to figure out the relation (\ref{srel}). And very specially Chuck Weibel for continuing support and encouragement.

\subsection{Computations with cohomological operations}
\llabel{s1}
\llabel{comp2}
For the purpose of this section a pointed smooth simplicial scheme is
a pointed simplicial scheme such that its terms are disjoint unions of
smooth schemes of finite type over $k$ pointed by a disjoint
point. For a pointed smooth simplicial scheme $\cal X$ the simplicial
suspension $S^1_s\wedge {\cal X}$ is again a pointed smooth simplicial
scheme. For a motivic cohomology class
$$\alpha\in H^{p,q}({\cal X},R)$$
of a pointed smooth simplicial scheme $\cal X$ we let 
$$\sigma_s\alpha\in H^{p,q}(S^1_s\wedge {\cal X},R)$$
denote the simplicial suspension of $\alpha$. The goal of this section
is to prove the following uniqueness result.
\begin{theorem}
\llabel{maincomp2} Let $k$ be a field of characteristic zero. Let
$\phi_i$, $i=1,2$ be two cohomological operations on the motivic
cohomology of pointed smooth simplicial schemes of the form
$$\tilde{H}^{2n+1,n}(-,\zz/l)\sr \tilde{H}^{2nl+2,nl}(-,\zz/l)$$
such that:
\begin{enumerate}
\item for $b\in\zz/l$ one has $\phi_i(b\alpha)=b\phi_i(\alpha)$
\item for any $\alpha\in H^{2n,n}({\cal X},\zz/l)$ one has
$\phi(\sigma_s\alpha)=0$
\end{enumerate}
Then there exists $c\in \zz/l$ such that $\phi_1=c\phi_2$.
\end{theorem}
Observe first that since motivic cohomology respect local equivalences
and any pointed simplicial sheaf is locally equivalent to a pointed
smooth simplicial scheme, operations $\phi_i$ extend canonically to
operations on the motivic cohomology of pointed simplicial sheaves.

Let $K_{m}$, $m=2n,2n+1$ be a pointed simplicial sheaf which
represents on the pointed motivic homotopy category the functor
$\tilde{H}^{m,n}(-,\zz/l)$ and $\alpha_m$ be the canonical class in
$\tilde{H}^{m,n}(K_{m},\zz/l)$.

Since both operations $\phi_i$ are natural for morphisms of pointed
smooth simplicial schemes and any morphism in the motivic homotopy
category can be represented by a hat of morphisms of pointed smooth
simplicial schemes it is sufficient to show that 
$$\phi_1(\alpha_{2n+1})=c\phi_2(\alpha_{2n+1}).$$
for an element $c\in\zz/l$.
\begin{lemma}
\llabel{later3}
For all $i>0$ one has $\alpha_{2n}^i\ne 0$.
\end{lemma}
\begin{proof}
Since $K_{2n}$ represents the functor $\tilde{H}^{2n,n}(-,\zz/l)$ the
condition $\alpha_{2n}^i=0$ would imply that for any $X$ and any
$\alpha\in H^{2n,n}(X,\zz/l)$ one has $\alpha^i=0$. Taking $X$ to be
${\bf P}^N$ for $N$ large enough and $\alpha$ to be a generator of
$H^{2n,n}({\bf P}^N,\zz/l)$ we get a contradiction.
\end{proof}
\begin{lemma}
\llabel{later2} Let $k$ be a field of characteristic zero. Then the
Kunnet homomorphism
$$\tilde{H}^{*,*}(K_{2n},\zz/l)\oo_{H^{*,*}}\dots\oo_{H^{*,*}}\tilde{H}^{*,*}(K_{2n},\zz/l)\sr
\tilde{H}^{*,*}(K_{2n}^{\wedge i},\zz/l)$$
is an isomorphism for all $i\ge 0$.
\end{lemma}
\begin{proof}
The Kunnet homomorphism is an isomorphism for all spaces whose motives are direct sums of Tate motives. In particular it is an isomorphism for $K_{2n}$ which is a direct sum of Tate motives by \maintwo.
\end{proof}

Choosing $K_m$ to be a sheaf of $\zz/l$ vector spaces we get an action
of $Aut(\zz/l)=(\zz/l)^*$ by automorphisms on $K_{m}$. This action
defines an action on the motivic cohomology of $K_{m}$ with
$\zz/l$-coefficients which gives a canonical splitting of these
cohomology groups into the direct sum of subspaces of weights
$0,\dots,l-2$. To distinguish the weight in this sense from the
weight as the second index of motivic cohomology we will call the
former one the scalar weight and specify it by a third index such that
$H^{p,q,r}(K_{m},\zz/l)$ is the subgroup of elelemnts of scalar
weight $r$ in $H^{p,q}(K_{m},\zz/l)$. A class $\gamma$ is in this
subgroup if for any $a\in (\zz/l)^*$ the automorphism $f_a$ defined by
$a$ takes $\gamma$ to $a^r\gamma$.

For an element $x$ in $H^{*,*}(K_m,\zz/l)$ we let $s(x)$
(resp. $w(x)$, $d(x)$) denote its scalar weight (resp. its motivic
weight, its dimension) if it is well defined. 

\begin{lemma}
\llabel{newlemma}
Let $0\le s\le l-2$ and let $x\in H^{*,*,s}(K_n,\zz/l)$, $x\ne 0$. Then one has:
\begin{eq}
\llabel{newineq1}
w(x)\ge\left\{
\begin{array}{ll}
sn&\mbox{\rm if $s>0$}\\
(l-1)n&\mbox{\rm if $s=0$}
\end{array}
\right.
\end{eq}
If $n>0$ and the equality holds in (\ref{newineq1}) then there is $c\in \zz/l$ such that
\begin{eq}
\llabel{newineq2}
x=\left\{
\begin{array}{llll}
c\alpha^s_{2n}&{\rm or}&c(\beta\alpha_{2n})\alpha_{2n}^{s-1}&\mbox{\rm if $s>0$}\\
c\alpha^{l-1}_{2n}&{\rm or}&c(\beta\alpha_{2n})\alpha_{2n}^{l-2}&\mbox{\rm if $s=0$}
\end{array}
\right.
\end{eq}
where $\beta$ is the Bockstein homomorphism. 
\end{lemma}
\begin{proof}
We may assume that $n>0$. Then by \anotherth we have
$$H^{*,w,s}(K_n,\zz/l)=\bigoplus_{m\ge 1, m\equiv s\, mod\, (l-1)}Hom_{DM}(S^m_{tr}(\zz/l(n)[2n]\oplus\zz/l(n)[2n+1]),\zz/(w)[*])$$
By \maintwo  one has
$$Hom_{DM}(S^m_{tr}(\zz/l(n)[2n]\oplus\zz/l(n)[2n+1]),\zz/(w)[*])=0$$
for 
$$w<(\sum m_i)n+(\sum i m_i)(l-1)$$
where $m=\sum m_i l^i$, $0\le m_i\le l-1$.

If $s>0$ we have
$$sn\le (\sum m_i)n+(\sum i m_i)(l-1)$$
for any $m$ such that $m\equiv s\, mod\, (l-1)$ since $\sum m_i\equiv m\, mod\, (l-1)$. If $s=0$ we have $\sum m_i\equiv 0\, mod\, (l-1)$ and since $\sum m_i>0$ we conclude that $\sum m_i\ge l-1$ and we get 
$$(l-1)n\le (\sum m_i)n+(\sum i m_i)(l-1).$$
An equality may be achieved only if $\sum i m_i=0$ i.e. if $m<l$. For $m<l$ we have
$$S^m_{tr}(\zz/l(n)[2n]\oplus\zz/l(n)[2n+1])=\zz/l(nm)[2nm]\oplus \zz/l(nm)[2nm+1]$$
and it is easy to see that the corresponding motivic cohomology classes of $K_n$ are $\alpha_{2n}^m$ and $(\beta\alpha_{2n})\alpha_{2n}^{m-1}$. For $s>0$ we have $m=s$ and for $s=0$ we have $m=l-1$ which finishes the proof.
\end{proof}
\begin{lemma}
\llabel{newlemma2}
As a $H^{*,*}(Spec(k))$-module, $H^{*,*}(K_n,\zz/l)$ is generated by classes $x$ such that $d(x)\ge 2w(x)$.
\end{lemma}
\begin{proof}
It follows immediately from \anotherth, \maintwo and the definition of a proper Tate object (loc. cit.).
\end{proof}

The first condition of the theorem means that
$$\phi_i(\alpha_{2n+1})\in \tilde{H}^{2nl+2,nl,1}(K_{2n+1},\zz/l)$$
The second condition says that $\phi_i(\alpha_{2n+1})$ lie in the
kernel of the homomorphism
$$\tilde{H}^{2nl+2,nl}(K_{2n+1},\zz/l)\sr
\tilde{H}^{2nl+2,nl}(\Sigma^1_sK_{2n},\zz/l))$$
defined by the obvious morphism
\begin{eq}
\llabel{mori}
i:\Sigma^1_sK_{2n}\sr K_{2n+1}.
\end{eq}
The statement of the theorem follows now from the proposition below.
\begin{proposition}
\llabel{ker2}
The kernel of the homomorphism 
\begin{eq}
\llabel{mori2}
\tilde{H}^{2nl+2,nl,1}(K_{2n+1},\zz/l)\sr
\tilde{H}^{2nl+2,nl}(\Sigma^1_sK_{2n},\zz/l)
\end{eq}
is generated by one element.
\end{proposition}
\begin{proof}
We can choose $K_{2n}$ to be a sheaf of abelian groups. Then we may
realize $K_{2n+1}$ as the simplicial sheaf $B_{\bullet}K_{2n}$ where
$B_{\bullet}$ refers to the standard simplicial classifying space of a
group space such that 
$$B_p(K_{2n})=K_{2n}^p.$$
Let $M(w)$ be fibrant (injective) model for the complex
$\zz/l(w)$. The complexes $\tilde{H}^0(B_pK_{2n},M(w))$ form a
cosimplicial complex and we let 
$$N\tilde{H}^0(B_*K_{2n},M(w))$$
denote the corresponding normalized bicomplex. Note that its terms
along the former cosimplicial dimension are of the form
$\tilde{H}^0(K_{2n}^{\wedge p},M(w))$. Then we have
$$\tilde{H}^{d,w}(K_{2n+1},\zz/l)=H^{d}(Tot(N\tilde{H}^{0}(B_{\bullet}K_{2n},M(w))))$$
where $Tot$ refers to the total complex of our bicomplex. Hence we
have a standard spectral sequence of a bicomplex with the $E_1$ term
of the form
\begin{eq}
\llabel{specseq} E_1^{p,q}=H^{q}(N\tilde{H}^{0}(B_*K_{2n},
M(w))_p)=\tilde{H}^{q,w}(K_{2n}^{\wedge p},\zz/l)
\end{eq}
which tries to converge to $\tilde{H}^{p+q,w}(K_{2n+1},\zz/l)$. To
keep track of the motivic weight of our cohomology groups we will use
a third index $E^{p,q,w}_r$ for the terms of this spectral sequence. 

One can easily see that this spectral sequence coincides with the
spectral sequence defined by the skeletal filtration
\begin{eq}
\llabel{skfiltr}
sk_0(B\BB K_{2n})\subset sk_1(B\BB K_{2n})\subset\dots\subset
sk_p(B_{\bullet}K_{2n})\subset\dots
\end{eq}
on the simplicial sheaf $B\BB K_{2n}$. Note that the first term of
this filtration $sk_1 B_{\bullet}K_n$ is $\Sigma^1_sK_n$ and the
morphism (\ref{mori}) is the natural inclusion
$$i:sk_1
B_{\bullet}K_n\sr B_{\bullet}K_n.$$
\begin{lemma}
\llabel{conv}
The spectral sequence (\ref{specseq}) converges to
$\tilde{H}^{p+q,w}(K_{2n+1},\zz/l)$. 
\end{lemma}
\begin{proof}
Interpreting (\ref{specseq}) as the spectral sequence associated with
the filtration (\ref{skfiltr}) we see that to prove the convergence it
is enough to show that for a given $w$ there exists $N$ such that for
all $p>N$ one has
$$\tilde{H}^{*,w}(sk_{p}(B_{\bullet}K_{2n})/sk_{p-1}(B_{\bullet}K_{2n}),\zz/l)=0.$$ 
It is easy to see that we have
$$sk_{p}(B_{\bullet}K_{2n})/sk_{p-1}(B_{\bullet}K_{2n})=\Sigma^{p}_s
K_{2n}^{\wedge p}$$
where $\Sigma_s$ is the simplicial suspension. On the other hand by
\cite[Cor. 3.4]{Redpub} we know that $K_{2n}$ is n-fold $T$-connected and
therefore $K_{2n}^{\wedge p}$ is $np$-fold $T$-connected and its
motivic cohomology of weight $<np$ are zero.
\end{proof}
Let us consider now what the spectral sequence (\ref{specseq}) says
about the group $A=\tilde{H}^{2nl+2,nl,1}(K_{2n+1},\zz/l)$.
Note first that since the spectral sequence is constructed out of a
filtration which respects the action of $Aut(\zz/l)$ it splits into a
direct sum of spctral sequences $E_r^{p,q,w,s}$ for individual scalar
weights $s=0,\dots,l-2$. Hence the groups which contribute to $A$ are of
the form
\begin{eq}
\llabel{e1descr}
E_1^{p,2nl+2-p,nl,1}=\tilde{H}^{2nl+2-p,nl,1}(K_{2n}^{\wedge p},\zz/l)
\end{eq}

\begin{lemma}
\llabel{actneed1}
For any $p>1$, $q<nl$ one has
$$\tilde{{H}}^{*,q,1}(K^{\wedge p}_{2n},\zz/l)=0$$
\end{lemma}
\begin{proof}
By Lemma \ref{later2} it is sufficient to consider elements of the form
$x=x_1\oo\dots\oo x_p$ where $x_i$ are elements of
$\tilde{H}^{*,*}(K_{2n},\zz/l)$ with a well defined scalar weight. Suppose that $s(x)=1$. Since $p>1$ there are two
possibilities. Either $s(x_i)=0$ for some $i$ or
\begin{eq}
\llabel{ineq}
s(x_1)+\dots+s(x_p)\ge l
\end{eq}
In the first case we may assume without loss of generality that $s(x_1)=0$. Then by Lemma \ref{newlemma} $w(x_1)\ge (l-1)n$ and since $w(x_2)\ge n$ we conclude that  $w(x)\ge nl$. In the second case Lemma \ref{newlemma} implies that $w(x)=\sum w(x_i)\ge (\sum s(x_i))n\ge n$.
\end{proof}
\begin{lemma}
\llabel{actneed2} For any $p\ge 3$ one has
$$\tilde{{H}}^{2nl+2-p,nl,1}(K^{\wedge p},\zz/l)=0$$
\end{lemma}
\begin{proof}
By Lemma \ref{later2} it is sufficient to consider elements of the form  $a\, x_1\oo\dots\oo x_p$ where $a\in H^{d,v}(Spec(k))$ and
$$x=x_{1}\oo\dots\oo x_{p}\in H^{2nl+2-p-d,nl-v,1}(K_{2n}^{\wedge p},\zz/l)$$
By Lemma \ref{newlemma2} we may further assume that $d(x_i)\ge 2w(x_i)$.  By Lemma \ref{actneed1} we
conclude that for $v=0$. Since $H^{d,0}(Spec(k))=0$ for $d<0$ and $p>2$ this shows that $x=0$.
\end{proof}
Lemma \ref{actneed2} together with (\ref{e1descr}) show that there is
a short exact sequence
$$0\sr E_{\infty}^{2,2nl,nl,1}\sr
\tilde{H}^{2nl+2,nl,1}(K_{2n+1},\zz/l)\sr E_{\infty}^{1,2nl+1,nl,1}\sr
0$$
For $p=1$ the incoming differentials are zero starting with $d_1$ and
hence $E_{\infty}$ is contained in $E_1$ and we have an exact sequence
$$0\sr E_{\infty}^{2,2nl,nl,1}\sr
\tilde{H}^{2nl+2,nl,1}(K_{2n+1},\zz/l)\sr
\tilde{H}^{2nl+1,nl,1}(K_{2n},\zz/l)$$
where the last arrow is exactly (\ref{mori2}). It remains to show that 
$E_{\infty}^{2,2nl,nl,1}$ is generated by one element. Since this is a
subgroup of the corresponding $E_2$ term it is sufficient to show that
this $E_2$ term is generated by one element. 

The $E_2^{p,q,nl,s}$ term is the cohomology of the complex
$$\tilde{H}^{q,nl,s}(K_{2n}^{\wedge(p-1)},\zz/l)\sr
\tilde{H}^{q,nl,s}(K_{2n}^{\wedge p},\zz/l)\sr
\tilde{H}^{q,nl,s}(K_{2n}^{\wedge (p+1)},\zz/l)$$
where the differential is defined by the differential in the normalized
complex corresponding to $B\BB K_{2n}$. 
\comment{Replacing it by
the quasi-isomorphic complex obtained by taking alternated sums of
morphisms defined by face maps instead of normalization we see that
our $E_2$ term can be computed as cohomology of the complex
$$\tilde{H}^{q,nl,s}(K_{2n}^{p-1},\zz/l)\sr
\tilde{H}^{q,nl,s}(K_{2n}^{p},\zz/l)\sr
\tilde{H}^{q,nl,s}(K_{2n}^{p+1},\zz/l)$$
where the differential is given by alternated sums of morphisms defined
by face maps in $B\BB K_{2n}$.} 
\begin{lemma}
\llabel{actneed3} 
For $p>1$ the group
$$D_p=\tilde{H}^{2nl,nl,1}(K_{2n}^{\wedge p},\zz/l)$$
is a free $\zz/l$ module generated by monomials of the form
$$\alpha_{2n}^{i_1}\wedge\dots\wedge\alpha_{2n}^{i_p}$$
where $i_j>0$ and $\sum_j i_j=l$. 
\end{lemma}
\begin{proof}
Note first that these monomials are linearly independent by Lemmas
\ref{later2} and \ref{later3}. It remains to show that they generate
$D_p$ as a $\zz/l$-module.  By Lemma \ref{later2} and Lemma \ref{newlemma2} we conclude that it is sufficient to consider elements of the form $x=a x_1\wedge\dots\wedge x_p$ where $a\in H^{*,*}(Spec(k))$ and 
$$\sum_i s(x_i)\equiv 1\, mod\, (l-1)$$
$$d(x_i)\ge 2w(x_i)$$
By Lemma \ref{actneed1} we conclude that $a\in H^{*,0}(Spec(k))$ and since $H^{>0,0}(Spec(k))=0$ and $H^{*,*}(Spec(k))=\zz/l$ we may assume that $a=1$. Now a series of elementary calculations based on Lemma \ref{newlemma} finish the proof.
\end{proof}
To proceed further we will use a techique which allows one to obtain
elements in the $E_2$ term of the spectral sequence associated with
the skeletal filtration on $B\BB G$ for any sheaf of groups $G$.
Let $v:G\times G\sr G$ be the morphism given by $(g_1,g_2)\mapsto g_1
g_2^{-1}$. Note that the face map
$$\partial_i:G^{p+1}\sr G^p$$
in $B\BB G$ is of the form
$$
\partial_i(g_0,\dots,g_{p})=\left\{
\begin{array}{ll}
(g_0,\dots,\hat{g}_i,\dots,g_p)&\mbox{\rm  for $i\le p$}\\
(g_0g_p^{-1},\dots, g_{p-1}g_{p}^{-1})&\mbox{\rm  for $i=p$}
\end{array}
\right. 
$$
Let $\gamma$ be an element in $H^{d,w}(G,\zz/l)$ such that
\begin{eq}
\llabel{gammacond}
v^*(\gamma)=\gamma\wedge 1 - 1\wedge \gamma.
\end{eq}
Consider the pointed simplicial scheme $B\BB {\bf G}_a$ over $\zz/l$
and let 
$$C^{\bullet}={\cal O}(B\BB {\bf G}_a)$$
be the corresponding (reduced) cosimplicial abelian group. Then
$C^0=0$ and for $p>0$ the terms of $C^{\bullet}$ are polynomial rings
$$C^p=\zz/l[x_1,\dots,x_p]$$
and the face maps are given by obvious explicit formulas. Note that
the face map are homogenious in $x_i$ of degree $1$ and therefore we
may consider $C^{\bullet}$ as a graded simplicial abelian group. We
will write this grading by degrees in $x_i$'s as the second index.

Define homomorphisms
$$C^{p,q}\sr H^{dq,wq}(G^p,\zz/l)$$
by the rule $x_i\mapsto 1\wedge\dots\wedge\gamma\wedge\dots\wedge 1$
where $\gamma$ is on the $i$-th place. One verifies immediately that
our condition on $v^*(\gamma)$ implies that these homomorphisms define
a homomorphism of complexes
\begin{eq}
\llabel{homc}
\tilde{C}^{*,q}\sr E_1^{*,dq,wq}
\end{eq}
where $\tilde{C}^*$ is the normalized complex defined by the
cosimplicial abelian group $C^{\bullet}$ and $E_1^{*,dq,wq}$ is the
appropriate row of our spectral sequence for $B\BB G$ with $d_1$ as
the differential. The cohomology of $\tilde{C}^*$ are the cohomology
groups $H^{*}(B{\bf G}_a, {\bf G}_a)$ over $\zz/l$. Hence, any
$\gamma$ as above defines a homomorphism
\begin{eq}
\llabel{cohcoh}
H^{p,q}(B{\bf G}_a, {\bf G}_a)\sr E_2^{p,dq,wq}
\end{eq}
where the second grading on the left hand side is defined by the
polynomial degree of the cocycles. 

Let us return now to the case when $G=K_{2n}$ and
$\gamma=\alpha_{2n}$. Note that the condition (\ref{gammacond}) is
satisfied since $\alpha_{2n}$ is defined by the identity homomorphism
of the abelian group $K_{2n}$ and hence its composition with
$v:K_{2n}\times K_{2n}\sr K_{2n}$ is exactly $\alpha_{2n}\oo
1-1\oo\alpha_{2n}$. Since $\gamma$ is homogenious of degree $1$ with
respect to the scalar weight the homomorphism (\ref{homc}) in this
case is of the form
\begin{eq}
\llabel{homc2}
\tilde{C}^{*,q}\sr E_1^{*,2nq,nq,q\, mod\, (l-1)}
\end{eq}
The part of this homomorphism we are interested in at the moment is 
\begin{eq}
\llabel{homc2.5}
\tilde{C}^{p,l}\sr E_1^{p,2nl,nl,1}
\end{eq}
Lemma \ref{actneed3} implies immedialtely that (\ref{homc2.5}) is an
isomorphism for $p>1$. Therefore, the corresponding map
\begin{eq}
\llabel{homc3}
H^{p,l}(B{\bf G}_a, {\bf G}_a)\sr E_2^{p,2nl,nl,1}
\end{eq}
is surjective for $p=2$ and is an isomorphism for $p>2$. It remains to
show that for $p=2$ the left hand side of (\ref{homc3}) is generated
by one element. This follows immediately from the computation of
$H^*(B{\bf G}_a, {\bf G}_a)$ given in \cite[Th.12.1, p.375]{Lazard}.
\begin{remark}\rm
Note that if (\ref{gammacond}) is
satisfied for an element $\gamma$ then it is also satisfied for
$u(\gamma)$ for any motivic Steenrod operation $u$. Hence we can
extend homomorphism (\ref{cohcoh}) to a homomorphism
\begin{eq}
\llabel{cohcohst} {\cal A}^{a,b}\oo_{\zz/l} H^{p,q}(B{\bf G}_a, {\bf
G}_a)\sr E_2^{p,a+dq,b+wq}
\end{eq}
\end{remark}

\end{proof}

\subsection{Computations with symmetric powers}
\llabel{s2}
In this section we fix a prime $l$ and consider the categories of
motives with coefficients in ${R}$ where $R$ is a commutative ring
such that all primes but $l$ are invertible in ${R}$. For our
applications we will need the cases of $R=\zz_{(l)}$ and
$R=\zz/l$. Our goal is to prove several results about the structure of
the symmetric powers $S^i(M)$ for $i<l$ when $M$ is a Tate motives of
the form
$${{R}}(p)[2q]\sr M\sr {{R}}\sr {{R}}(p)[2q+1]$$
and to use these results to define a cohomological operation
$$\phi_{l-1}:H^{2q+1,p}(-,R)\sr H^{2ql+2,pl}(-,R)$$
Let us first consider an arbitrary tensor additive category $C$ which
is $R$-linear and Karoubian (has images of projectors). For any $i<l$
and any $M$ in $C$ define the symmetric power $S^i(M)$ as follows. The
symmetric group $S_i$ acts by permutations on $M^{\oo i}$. Since $i!$
is invertible in our coefficients ring we may consider the averaging
projector $p:M^{\oo i}\sr M^{\oo i}$ given by
$$p=(1/i^!)\sum_{\sigma\in S_i} \sigma$$
We set $S^i(M):=Im(p)$. We will use morphisms
$$a:S^i(M)\sr S^{i-1}(M)\oo M$$
and
$$b:S^{i-1}(M)\oo M\sr S^{i}(M)$$
where $a$ is defined as the quotient of the morphism $\tilde{a}:M^{\oo
i}\sr M^{\oo i}$ given by
$$\tilde{a}(m_1\oo\dots\oo m_i)=\sum_{j=1}^i (m_1\oo\dots\oo
\hat{m}_j\oo\dots\oo m_i)\oo m_j$$
and $b$ is the quotient of the identity morphism. 

Let us consider now the case when $C=DT({\cal X},R)$ for a smooth
simplicial scheme $\cal X$ and $M$ is a motive which is given together
with a distinguished triangle of the form
$${{R}}(p)[2q]\stackrel{x}{\sr} M\stackrel{y}{\sr}
{{R}}\stackrel{\alpha}{\sr} {{R}}(p)[2q+1]$$
where $p,q\ge 0$. Composing $a$ with the morphism defined by $y$ we
get a morphism 
$$u:S^{i}(M)\sr S^{i-1}(M)$$
and composing $b$ with the morphism defined by $x$ we get a morphism
$$v:S^{i-1}(M)(p)[2q]\sr S^{i}(M)$$
\begin{lemma}
\llabel{mainseq}
There exist unique morphisms
$$r:S^{i-1}(M)\sr {{R}}(ip)[2iq+1]$$
and 
$$s:{{R}}\sr S^{i-1}(M)(p)[2q+1]$$
such that the sequences 
\begin{eq}
\llabel{seq1}
{{R}}(ip)[2iq]\stackrel{x^i}{\sr}S^i(M)\stackrel{u}{\sr}
S^{i-1}(M)\stackrel{r}{\sr} {{R}}(ip)[2iq+1]
\end{eq}
\begin{eq}
\llabel{seq2} S^{i-1}(M)(p)[2q]\stackrel{v}{\sr}
S^{i}(M)\stackrel{y^i}{\sr}{{R}}\stackrel{s}{\sr}
S^{i-1}(M)(p)[2q+1]
\end{eq}
are distinguished triangles. If $p>0$ then these triangles are
isomorphic to the triangles
$$\Pi_{\ge ip}(S^i(M))\sr S^i(M)\sr \Pi_{< ip}(S^i(M))\sr
\Pi_{\ge ip}(S^i(M))[1]$$
and
$$\Pi_{\ge p}(S^i(M))\sr S^i(M)\sr \Pi_{<p}(S^i(M))\sr
\Pi_{\ge p}(S^i(M))[1]$$
\end{lemma}
\begin{proof}
Assume first that $p>0$. Since the category of Tate motives is closed
under tensor products and direct summands the symmetric power of a
Tate motive is a Tate motive. Therefore it is sufficient to verify
that the first three terms of the sequences (\ref{seq1}) and
(\ref{seq2}) satisfy the conditions of \refunext for $n=ip$
and $n=p$ respectively.

By \refsltensor one has
$$s_*(M^{\oo i})=s_*(M)^{\oo i}$$
which immediately implies that
$$s_*(S^i(M))=S^i(s_*(M))$$
and that these isomorphisms are compatible with the maps $a,b$.  Since
$p>0$ we have $s_*(M)={{R}}\oplus {{R}}(p)[2q]$ and therefore
$$s_*(S^i(M))=\oplus_{j=0}^i {{R}}(pj)[2qj].$$
where the morphism ${R}(pj)[2qj]\sr s_*(S^i(M))$ is $s_*(x^j)$. We
denote this morphism by $t^j$. Computing the slices of the morphisms
involved in (\ref{seq1}) and (\ref{seq2}) one gets:
\begin{eq}
\llabel{foru}
s_*(u)(t^j)=(i-j)t^j
\end{eq}
\begin{eq}
\llabel{forv} 
s_*(v)(t^j)=t^{j+1}
\end{eq}
The morphism $x^i$ is $t^i$. Since $i-j$ are invertible for all
$j=0,\dots, i-1$ this implies together with (\ref{foru}) that
(\ref{seq1}) satisfies the conditions of \refunext. The
morphism $y^i$ takes $t^j$ to $0$ for $j\ne 0$ and takes $1$ to
$1$. This implies together with (\ref{forv}) that (\ref{seq2})
satisfies the conditions of \refunext.

Consider now the case of $p=0$. Using \refdt we can
identify $DT_0$ with a full sybcategory in $DLC({\cal X},R)$. If $q>0$
consider the homology of $S^i(M)$ with respect to the standard
t-structure on $DLC({\cal X},R)$. One can easily see that $x^i$
defines an isomorphism of $R[2iq]$ with $\tau_{\ge 2iq}(S^i(M))$ and $u$
defines an isomorphism of $S^{i-1}(M)$ with $\tau_{< 2iq}(S^i(M))$
where $\tau$ refers to the canonical filtration with respect to our
t-structure. The standard argument shows now that there exists a
unique $r$ with the required property. A similar argument shows the
existence and uniqueness of $s$.

Conisder now the case $p=q=0$. Then the original triangle comes from
an exact sequence of the form
\begin{eq}
\llabel{shortex}
0\sr R\sr M\sr R\sr 0
\end{eq}
in $LC({\cal X})$ and for all $i<l$ we have $S^i(M)\in LC$. To prove
the existence and uniqueness of $r$ and $s$ it is sufficient to show
that the sequences defined by $x^i$ and $u$ and by $v$ and $y^i$ are
exact. We can verify the exactness on each term of $\cal X$
individually. On a smooth scheme the constant presheaf with transfers
is a projective object and therefore the restrictions of
(\ref{shortex}) to each term of $\cal X$ are split exact. The
exactness of the sequences defined by $x^i$ and $u$ and by $v$ and
$y^i$ follows by an easy computation.
\end{proof}
Consider the composition 
$$(r\oo Id_{{{R}}(p)[2q+1]})\circ s:{{R}}\sr {{R}}((i+1)p)[2(i+1)q+2]$$
Since the morphism $\alpha:{{R}}\sr{{R}}(p)[2q+1]$ determines $M$ up to an
isomorphism which commutes with $x$ and $y$ and our construction is
natural with respect to such morphisms in $M$, this composition
depends only on $\alpha$. We denote it by $\phi_i(\alpha)$. Note that
it is defined only for $i<l$. Since our construction is natural in $M$
and the inverse image functors commute with tensor product we get the
following result.
\begin{lemma}
\llabel{phinat}
For any $\alpha\in H^{2q+1,p}({\cal X},R)$ and any morphism of
simplicial schemes $f:{\cal Y}\sr {\cal X}$ one has
$$f^*(\phi_i(\alpha))=\phi_i(f^*(\alpha))$$
\end{lemma}
\begin{remark}\rm
One observes easily that $\phi_1(\alpha)=\alpha^2$. One can aslo show
that $\phi_i(\alpha)=0$ for $i<l-1$. We will see in Lemma
\ref{isnotzero} that for $R=\zz/l$ and any $n\ge 0$ the operation
$\phi_{l-1}$ is not identically zero.
\end{remark}
\begin{proposition}
\llabel{prod}
Let $\gamma$ be a morphism of the form ${{R}}\sr{{R}}(r)[2s]$ and
$\sigma$ a morphism of the form ${{R}}\sr{{R}}(p)[2q+1]$. Then one has
$$\phi_i(\gamma\sigma)=\gamma^{i+1}\phi_i(\sigma)$$
\end{proposition}
\begin{proof}
Set $\alpha=\gamma\sigma$. For simplicity of notations we will write
$\{n\}$ instead of $(r)[2s]$ and $\{m\}$ instead of $(p)[2q+1]$. For
example $X\{i(n+m)\}$ is $X(i(r+p))[i(2s+2q+1)]$. 

Let $M_{\gamma}$ and $M_{\sigma}$ be objects defined
(up to an isomorphism) by distinguished triangles
$${{R}}\{n\}[-1]\sr M_{\gamma}\sr {{R}}\stackrel{\gamma}{\sr}{{R}}\{n\}$$
$${{R}}\{m\}[-1]\sr M_{\sigma}\sr {{R}}\stackrel{\sigma}{\sr}{{R}}\{m\}$$
The octahedral axiom applied to the representation of $\alpha$ as
compositions
$${{R}}\stackrel{\gamma}{\sr}{{R}}\{n\}\stackrel{\sigma\{m\}}{\longrightarrow}
{{R}}\{n+m\}$$
and
$${{R}}\stackrel{\sigma}{\sr}{{R}}\{m\}\stackrel{\gamma\{n\}}{\longrightarrow}
{{R}}\{n+m\}$$
shows that there are morphisms
$$f:M_{\sigma}\sr M_{\alpha}$$
$$g:M_{\alpha}\sr M_{\sigma}\{n\}$$
which fit into morphisms of distinguished triangles of the form
\begin{eq}
\llabel{diag1}
\begin{CD}
{{R}}\{m\}[-1] @>>> M_{\sigma} @>>> {{R}} @>\sigma>>{{R}}\{m\}\\
@V\gamma\{m\}[-1]VV @VfVV @VIdVV @V\gamma\{m\}VV\\
{{R}}\{m+n\}[-1] @>>> M_{\alpha} @>>> {{R}} @>\alpha>>{{R}}\{m+n\}
\end{CD}
\end{eq}
\begin{eq}
\llabel{diag2}
\begin{CD}
{{R}}\{m+n\}[-1] @>>> M_{\alpha} @>>> {{R}} @>\alpha>>{{R}}\{m+n\}\\
@VIdVV @VgVV @V\gamma VV @VIdVV\\ 
{{R}}\{m+n\}[-1] @>>> M_{\sigma}\{n\}
@>>> {{R}}\{n\} @>\sigma\{n\}>>{{R}}\{m+n\}
\end{CD}
\end{eq}
Applying May's axiom \cite[Axiom TC3]{MayTT} to these two triangles we
conclude that morphisms $f$ and $g$ can be chosen in such a way that 
\begin{eq}
\llabel{compose}
g\circ f=Id\oo \gamma
\end{eq}
Consider now the diagrams
$$
\begin{CD}
S^{i}(M_{\sigma}) @>>> {{R}} @>>>
S^{i-1}(M_{\sigma})\{m\} @>>> S^{i}(M_{\sigma})[1]\\
@VS^i(f)VV @VIdVV
@VS^{i-1}(f)\oo \gamma \{m\}VV @VS^i(f)[1]VV\\
S^{i}(M_{\alpha}) @>>> {{R}} @>>>
S^{i-1}(M_{\alpha})\{n+m\} @>>> S^{i}(M_{\alpha})[1]\\
\end{CD}
$$
and
$$
\begin{CD}
S^i(M_{\alpha}) @>>> S^{i-1}(M_{\alpha}) @>>>
{{R}}\{i(m+n)\}[1-i] @>>> \dots\\
@VS^i(g)VV @VS^{i-1}(g)\oo \gamma VV @VIdVV @VS^i(g)[1]VV\\
S^i(M_{\sigma})\{in\} @>>>
S^{i-1}(M_{\sigma})\{in\}  @>>>
{{R}}\{i(m+n)\}[1-i] @>>> \dots
\end{CD}
$$
Where:
\begin{enumerate}
\item the upper row in the first diagram is (\ref{seq2}) for
$M_{\sigma}$
\item the lower row in the first diagram is (\ref{seq2}) for
$M_{\alpha}$
\item the upper row in the second disgram is (\ref{seq1}) for
$M_{\alpha}$
\item the lower row in the second diagram is (\ref{seq1}) for
$M_{\sigma}$ twisted by $\{in\}$
\end{enumerate}
Let us show that these diagrams commute. The commutativity of the right
square in the first diagram is an immediatel corollary of the
commutativity of the left square in (\ref{diag1}). Since both rows are
dsitinguished triangles we conclude that there is a morphism
$$\psi:{{R}}\sr{{R}}$$
which makes two other squares commute. Applying the slice functor we
conclude that the commutativity of the left square implies that
$\psi=1$.

The commutativity of the left square in the second diagram is an
immediatel corollary of the commutativity of the middle square in
(\ref{diag2}). Since both rows are dsitinguished triangles we conclude
that there is a morphism 
$$\psi:{{R}}\{i(m+n)\}[-i]\sr{{R}}\{i(m+n)\}[-i]$$
which makes two other squares commute. Applying the slice functor we
conclude that the commutativity of the middle square implies that
$\psi=1$.

We see now that $\phi_i(\alpha)$ is the composition:
$$
\begin{CD}
{{R}} @>(1)>> S^{i-1}(M_{\sigma})\{m\}\\
@.  @VVS^{i-1}(f)\{m\}V\\
@. S^{i-1}(M_{\alpha})\{n+m\}\\
@. @VVS^{i-1}(g)\oo\gamma\{m+n\}V\\
@. S^{i-1}(M_{\sigma})\{(i+1)n+m\} @>(2)\{(i+1)n+m\}>>
{{R}}\{(i+1)(m+n)\}[1-i]
\end{CD}
$$
We further have by definition
$$\phi_i(\sigma)=(2)\circ (1)$$
and by (\ref{compose}) we have 
$$S^{i-1}(g)\circ S^{i-1}(f)=S^{i-1}(g\circ f)=Id\oo
S^{i-1}(\gamma)=Id\oo \gamma^{i-1}$$
Taking the composition we get
$$\phi_i(\alpha)=\gamma^{i+1}\phi_i(\sigma)$$
\end{proof}
\begin{cor}
\llabel{scalar}
For any $\alpha:R\sr R(p)[2q+1]$ and any $c\in\zz$ one has
$$\phi_{i}(c\alpha)=c^{i+1}\phi_i(\alpha)$$
\end{cor}
Since operations $\phi_i$ are natural in $\cal X$ we can extend them
to reduced motivic cohomology groups of pointed simplicial schemes in
the usual way. We can further extend then to the reduced motivic
cohomology of pointed simplicial sheaves using the fact that any
simplicial sheaf has a weakly equivalent replacement by a smooth
simplicial scheme.
\begin{cor}
\llabel{susphi}
Let $\alpha$ be a class in $\tilde{H}^{2q,p}({\cal X},\zz/l)$. Then
$$\phi_i(\sigma_s\alpha)=0$$
\end{cor}
\begin{proof}
The pull-back of $\sigma_s\alpha$ with respect to the projection
$$(S^1_s\times{\cal X})_+\sr \Sigma^1_s{\cal X}$$
is the class $\sigma\wedge \alpha$ where $\sigma$ is the canonical
class in $H^{1,0}(S^1_s,\zz/l)$. Since the restriction homomorphism is
a monomorphism it is enough to show that $\phi_{i}(\sigma\wedge
\alpha)=0$. By Proposition \ref{prod} we have
$$\phi_{i}(\sigma\wedge \alpha)=\phi_{i}(\sigma)\wedge \alpha^{i+1}$$
The class $\phi_{i}(\sigma)$ lies in the group $H^{2,0}(S^1_s)=0$
which proves the corollary.
\end{proof}
\begin{lemma}
\llabel{isnotzero}
For any $n\ge 0$ there exists $\cal X$ and $\alpha\in H^{2n+1,n}({\cal
X},\zz/l)$ such that $\phi_{l-1}(\alpha)\ne 0$.
\end{lemma}
\begin{proof}
To show that there exists $\alpha\in H^{2n+1,n}$
such that $\phi_{l-1}(\alpha)\ne 0$ it is sufficient in view of Proposition
\ref{prod} to show that there exists $\alpha\in H^{1,0}$ such that
$\phi_{l-1}(\alpha)\ne 0$ and then consider $\alpha\gamma$ for an
appropriate $\gamma$ i.e. we may assume that $n=0$. In this case one
can take $\alpha$ to be a generator of 
$$H^{1,0}(K(\zz/l,1),\zz/l)=\zz/l$$
this generator is represented by the canonical extension 
$$0\sr \zz/l\sr M\sr \zz/l\sr 0$$
which corresponds to the standard 2-dimensional representation of
$\zz/l$ over $\zz/l$. The symmetric power $S^{l-1}(M)$ is given by the
regular representation $\zz/l[\zz/l]$ of $\zz/l$ over $\zz/l$ and
$\phi_{l-1}(\alpha)$ is the second extension represented by the exact
sequence
\begin{eq}
\llabel{ext2}
0\sr \zz/l\sr \zz/l[\zz/l]\stackrel{g}{\sr} \zz/l[\zz/l]\sr \zz/l\sr
0
\end{eq}
where the middle arrow is the multiplication by the generator of
$\zz/l$. Let $K$ be the complex given by the middle two terms of
(\ref{ext2}) with the last one placed in degree $0$. Then we have a
distinguished triangle
\begin{eq}
\llabel{trext2}
\zz/l[1]\sr K\sr \zz/l\stackrel{\phi_{l-1}(\alpha)}{\sr}\zz/l[2]
\end{eq}
Since $\zz/l[\zz/l]$ is a projective $\zz/l$-module we have
$$Hom(K,\zz/l[2])=0$$
where the morphisms are in the derived category. From the long exact
sequence associated with (\ref{trext2}) we conclude that the map
\begin{eq}
\llabel{issur}
H^0(\zz/l,\zz/l)\sr H^{2}(\zz/l,\zz/l)
\end{eq}
defined by $\phi_{l-1}(\alpha)$ is surjective. Since the right hand
side of (\ref{issur}) is not zero we conclude that
$\phi_{l-1}(\alpha)\ne 0$.
\end{proof}
\begin{theorem}
\llabel{maincomp}
Let $\alpha\in \tilde{H}^{2n+1,n}({\cal X},\zz/l)$ be a motivic cohomology
class. Then there exists $c\in (\zz/l)^*$ such that
\begin{eq}
\llabel{neednow}
\phi_{l-1}(\alpha)=c\beta P^{n}(\alpha)
\end{eq}
where $\beta$ is the Bockstein homomorphism and $P^n$ is the motivic
reduced power operation. 
\end{theorem}
\begin{proof}
The operation $\phi_{l-1}$ satisfies the conditions of Theorem
\ref{maincomp2} by Lemma \ref{phinat}, Corollary \ref{scalar} and
Corollary \ref{susphi}. The operation $\beta P^n$ satisfies the first
condition of Theorem \ref{maincomp2} because the motivic Steenrod
operations are additive. It satisfies the second condition since for
$\alpha\in H^{2n,n}$ one has
$$\beta P^n(\sigma_s\alpha)=\sigma_s\beta P^n(\alpha)=\sigma_s\beta \alpha^l=0$$
where the first equality follows from \cite[Lemma 9.2]{Redpub}, the
second equality from \cite[Lemma 9.8]{Redpub} and the third from
\cite[Eq. (8.1)]{Redpub}. We conclude that (\ref{neednow}) holds for
$c\in\zz/l$. Since $\beta P^n\ne 0$ by \cite[Cor. 11.5]{Redpub} and
$\phi_{l-1}\ne 0$ by Lemma \ref{isnotzero} we conclude that $c\ne 0$.
\end{proof}
\subsection{Motivic degree theorem}
\llabel{s3}
In this section we fix a prime $l$ and unless the opposite is
explicitly specified we always assume that all other primes are
invertible in the coefficient ring. In particular $\zz$ always means
$\zz_{(l)}$ - the localization of $\zz$ in $l$.

Recall from \cite{MCpub} that we let $s_d(X)$ denote the d-th Milnor
class of a smooth variety $X$. This class lies in $H^{2d,d}(X,\zz)$
and if $dim(X)=d$ one may consider the number $deg(s_d(X))$.  We say
that a smooth projective variety $X$ is a $\nu_{n}$-variety if
$dim(X)=l^{n}-1$ and
$$deg(s_{l^{n}-1}(X))\ne 0 (mod\,l^2)$$
In \cite{MCpub} we constructed for any smooth projective variety $X$ a
stable normal bundle $V$ on $X$ and a morphism 
\begin{eq}\llabel{tau}
\tau:T^N\sr Th_X(V)
\end{eq}
in the pointed $\af$-homotopy category which defines the degree map on
the motivic cohomology. Consider the cofibration sequence
\begin{eq}
\llabel{cofseq} T^{N}\stackrel{\tau}{\sr} Th_{X}(V)\stackrel{p}{\sr}
Th_X(V)/T^N\stackrel{\partial}{\sr} \Sigma^1_s T^N
\end{eq}
For $d=dim(X)>0$ the Thom class 
$$t\in \tilde{H}^{2N-2d,N-d}(Th_{X}(V),\zz)$$
restricts to zero on $T^N$ for the weight reasons and there exists a
unique class
$$\tilde{t}\in \tilde{H}^{2N-2d,N-d}(Th_{X}(V)/T^N,\zz)$$
such that $p^*(\tilde{t})=t$. On the other hand the pull-back of the
tautological class in $H^{2N+1,N}(\Sigma^1_s T^N,\zz)$ with respect to
$\partial$ defines a class
$$v\in \tilde{H}^{2N+1,N}(Th_{X}(V)/T^N,\zz)$$
\begin{lemma}
\llabel{toplemma} Let $X$ be a smooth projective variety of dimension
$d=l^n-1$ where $n>0$. Then one has
\begin{eq}
\llabel{eq7}
Q_{n}(\tilde{t})=(deg(s_{l^n-1}(X))/l)v\,\, mod\,\, l
\end{eq}
\end{lemma}
\begin{proof}
Recall from \cite{Redpub} that $Q_n=\beta q_n\pm q_n\beta$ where $\beta$
is the Bockstein homomorphism. Since $\tilde{t}$ is the reduction of
an integral class $\beta(\tilde{t})=0$ and it is sufficient to show
that 
\begin{eq}
\llabel{eq72}
\beta q_n(\tilde{t})=(deg(s_{l^n-1}(X))/l)v\,\, mod\,\, l
\end{eq}
The image of
(\ref{cofseq}) in $DM$ is an appropriate twist of a sequence of the
form
\begin{eq}
\llabel{motseq} \zz(d)[2d] \stackrel{\tau'}{\sr} M(X)\sr
cone(\tau')\stackrel{v}{\sr} \zz(d)[2d+1]
\end{eq}
By \cite[Cor. 14.3]{Redpub} we have $q_n(t)=s_{l^n-1}(X)t$ and therefore there is
a commutative square in the motivic category of the form
$$
\begin{CD}
M(X) @>>> cone(\tau')\\
@Vs_{l^n-1}(X)VV @VVq_n(\tilde{t})V\\
\zz/l^2(d)[2d] @>>> \zz/l(d)[2d]
\end{CD}
$$
This square extends to a morphism of distinguished triangles
$$
\begin{CD}
\zz(d)[2d] @>\tau'>> M(X) @>>> cone(\tau') @>v>> \zz(d)[2d+1]\\
@VuVV @Vs_{l^n-1}(X)VV @VVq_n(\tilde{t})V @VVuV\\
\zz/l(d)[2d] @>>> \zz/l^2(d)[2d]  @>>>
\zz/l(d)[2d]  @>\beta>> \zz/l(d)[2d+1]
\end{CD}
$$
for some morphism $u$. If $u$ sends $1$ to $c$ then the commutativity
of the left square means that we have
$$deg(s_{l^n-1}(X))=lc\,\,mod\,\, l^2$$
and the commutativity of the right square means that we have 
$$cv=\beta q_n(\tilde{t})\,\, mod\,\, l$$
multiplying the second equality by $l$ and combining with the first
one we get
$$deg(s_{l^n-1}(X))v=l\beta q_n(\tilde{t})\,\, mod\,\, l^2$$
which is equivalent to (\ref{eq72}).
\end{proof}
\begin{remark}\rm
The intermediate statement (\ref{eq72}) of Lemma \ref{toplemma}
actually holds for any motivic Steenrod operation $\phi$ if one
replaces $s_{l^n-1}$ by an appropriate characteristic class $c_{\phi}$
as described in \cite[Th. 14.2]{Redpub}.
\end{remark}
From this point until the end of the section we consider all our
motives with $\zz/l$-coefficients.  In particular ``an embedded
simplicial scheme'' means a simplicial scheme embedded with respect to
$\zz/l$-coefficients. 

Recall that the Milnor operations $Q_i$ have the property that
$Q_i^2=0$ and we define for any pointed simplicial scheme $\cal X$ and
any $i\ge 0$ the motivic Margolis homology
$\widetilde{MH}^{*,*}_i({\cal X},\zz/l)$ of $\cal X$ as homology of
the complex $(\tilde{H}^{*,*}({\cal X},\zz/l),Q_i)$. Our first
application of Lemma \ref{toplemma} is the following result which is a
slight generalization of \cite[Th. 3.2]{MCpub}.
\begin{lemma}
\llabel{margolis} Let $\cal X$ be an embedded (with respect to
 $\zz/l$-coefficients) simplicial scheme such that there exists a
 $\nu_n$-variety $X$ with $M(X,\zz/l)$ in $DM_{\cal X}$. Let further
$$\tilde{\cal X}=cone({\cal X}_+\sr S^0)$$
be the unreduced suspension of $\cal X$. Then 
$$\widetilde{MH}^{*,*}_n(\tilde{\cal X},\zz/l)=0.$$
\end{lemma}
\begin{proof}
Our proof is a version of the proof given in \cite{MCpub}. We will
assume that $n>0$. The case $n=0$ has a similar (easier) proof. We
will use the notations established in the proof of Lemma
\ref{toplemma}. Let $cone(\tau')$ be the motive defined by
(\ref{motseq}). Consider the morphisms in $DM$ with $\zz/l$
coefficients of the form
$$M(\tilde{X})(d)[2d+1]\stackrel{Id\oo v}{\longleftarrow}M(\tilde{X})\oo
cone(\tau')\stackrel{Id\oo \tilde{t}}{\longrightarrow}M(\tilde{X})$$
Since $M(X)$ is in $DM_{\cal X}$, \refdonotneed shows that
$M(\tilde{X})\oo M(X)=0$ and therefore sequence (\ref{motseq}) implies
that the first arrow is an isomorphism. Consider the homomorphism
$$\phi:H^{*,*}(\tilde{X},\zz/l)\sr H^{*-2d-1,*-d}(\tilde{X},\zz/l)$$
defined by $(Id\oo \tilde{t})\circ (Id\oo v)^{-1}$. We claim that for
any motivic cohomology class $x$ of $\tilde{\cal X}$ one has
$$\phi Q_n(x)-Q_n\phi(x)=-(-1)^{deg(x)}s_{l^n-1}(X)x$$
which clearly implies the statement of the lemma. Since $Id\oo v$ is
an isomorphism it is sufficient to check that both sides become the
same after multiplication with $v$. Since $v$ is the image of a
morphism in the homotopy category it commutes with cohomological
operations and we have to check that
\begin{eq}
\llabel{thisneed}
Q_n(x)\wedge \tilde{t}-Q_n(x\wedge
\tilde{t})=-(-1)^{deg(x)}s_{l^n-1}(X)x\wedge v
\end{eq}
For $l>2$ we have
$$Q_n(x\wedge
\tilde{t})=Q_n(x)\wedge \tilde{t}+(-1)^{deg(x)} x\wedge
Q_n(\tilde{t})$$
by \cite[Prop. 13.3]{Redpub} and the same holds for $l=2$ by
\cite[Prop. 13.4]{Redpub} since $Q_i(\tilde{t})=0$ for $i<n$ by weight
reasons. Applying Lemma \ref{toplemma} we further get
$$Q_n(x\wedge
\tilde{t})=Q_n(x)\wedge \tilde{t}+(-1)^{deg(x)} x\wedge
v$$
which implies (\ref{thisneed}).
\end{proof}
Let $\cal X$ be an
embedded simplicial scheme, $n>0$ an integer and $X$ be a
$\nu_n$-variety such that $M(X)=M(X,\zz/l)$ lies in $DM_{\cal
X}(\zz/l)$.

Let $\zz/l_{\cal X}(i)[j]$ denote the Tate motives over $\cal X$ which
we identify with $M({\cal X},\zz/l)(i)[j]$. The image of (\ref{tau}) in
$DM(k,\zz/l)$ is a morphism of the form
$$\zz/l(d)[2d]\sr M(X)$$
and its composition with the morphism $\zz/l_{\cal X}(d)[2d]\sr
\zz/l(d)[2d]$ gives us relative fundamental class
$$\tau_{\cal X}:\zz/l_{\cal X}(d)[2d]\sr M(X)$$
On the other hand \refrl implies that the structure morphism
$\pi:M(X)\sr \zz/l$ is the composition of a unique morphism
$$\pi_{\cal X}:M(X)\sr \zz/l_{\cal X}$$
with the morphism $\zz/l_{\cal X}\sr \zz/l$.
\begin{theorem}
\llabel{degree} Consider a commutative diagram in $DM_{\cal X}(\zz/l)$
of the form
$$
\begin{CD}
M(X,\zz/l) @>s>> N\\
@V\pi_{\cal X}VV @VVrV\\
\zz/l_{\cal X} @>Id>> \zz/l_{\cal X}.
\end{CD}
$$
Assume that there exists a class $\alpha\in H^{p,q}({\cal X},\zz/l)$
such that the following conditions hold:
\begin{enumerate}
\item $p>q$ and $\alpha\ne 0$
\item $\alpha\circ r=0$
\item $Q_n(\alpha)=0$
\end{enumerate}
Then $s\circ \tau_{\cal X}:\zz/l_{\cal X}(d)[2d]\sr N$ is not zero.
\end{theorem}
\begin{proof}
Let $N'$ be the motive defined by the distinguished triangle
$$\zz/l_{\cal X}(q)[p-1]\sr N'\sr \zz/l_{\cal
X}\stackrel{\alpha}{\sr}\zz/l_{\cal X}(q)[p]$$
Our assumption that $\alpha\circ r=0$ is equivalent to the assumption
that there is a morphism $N\sr N'$ which makes the diagram
$$
\begin{CD}
N @>>> N'\\
@VrVV @VVV\\
\zz/l_{\cal X} @>Id>> \zz/l_{\cal X}
\end{CD}
$$
commutative. Therefore to prove the proposition it is sufficient to
show that the composition
\begin{eq}
\llabel{compnon}
g:\zz_{\cal X}(d)[2d]\sr M(X)\sr N\sr N'
\end{eq}
is non zero. We may now forget about the original $N$ and consider
only $N'$.

The composition $\pi_{\cal X}\tau_{\cal X}$ is zero and there exists a
unique morphism
$$\tilde{\pi}_{\cal X}:cone(\tau_{\cal X})\sr \zz/l_{\cal X}$$
which restricts to $\pi_{\cal X}$ on $M(X)$. If the
composition (\ref{compnon}) is zero then
$$\alpha\circ\tilde{\pi_{\cal X}}:cone(\tau_{\cal X})\sr \zz/l_{\cal
X}(q)[p]$$
is zero. To finish the proof of the proposition it remains to show
that it is non-zero. Smashing the sequence (\ref{cofseq}) with ${\cal
X}_{+}$ we get a cofibration sequence
$$T^{N}\wedge {\cal X}_+\sr Th_{X}(V)\wedge {\cal X}_+\sr
(Th_X(V)/T^N)\wedge{\cal X}_+\stackrel{\partial_{\cal X}}{\sr} \Sigma^1_s T^N\wedge
{\cal X}_+$$
Up to the shift of the bidegree by $(2N-2d,N-d)$, the motivic
cohomology of $(Th_X(V)/T^N)\wedge{\cal X}_+$ coincide as the module
over the motivic cohomology of ${\cal X}$ with the motivic cohomology
of $cone(\tau_{\cal X})$ such that $\tilde{\pi}_{\cal X}$ corresponds to the
pull-back of $\tilde{t}$.

Hence all we need to show that $\tilde{t}\alpha\ne 0$.  We are going
to show that $Q_n(\tilde{t}\alpha)\ne 0$. For $l>2$ one has by
\cite[Prop. 13.3]{Redpub} 
\begin{eq}
\llabel{prodf}
Q_n(u\wedge v)=Q_n(u)\wedge v \pm u\wedge Q_n(v)
\end{eq}
and since $Q_n(\alpha)=0$ we get that
\begin{eq}
\llabel{difpr}
Q_n(\tilde{t}\alpha)=Q_n(\tilde{t})\alpha.
\end{eq}
For $l=2$ we have additional terms in (\ref{prodf}) which depend on
$Q_{i}(\tilde{t})$ for $i<n$. It follows from the simple weight
considerations that $Q_{i}(\tilde{t})=0$ for $i<n$ and therefore
(\ref{difpr}) holds for $l=2$ as well.

Lemma \ref{toplemma} shows that the right hand side of (\ref{difpr})
equals $cv\alpha$ where $c=s_{l^n-1}(X)/l$. Since $X$ is a
$\nu_n$-variety, $c$ is an invertible element of $\zz/l$. Hence it
remains to check that $v\alpha\ne 0$. Since $v=\partial^*(u)$ where
$u$ is the generator of
$$\zz/l=H^{2N+1,N}(\Sigma^1_sT^N,\zz/l)$$
we have $v\alpha=\partial^*_{\cal X}(u\alpha)$.  The element $u\alpha$
lies in the bidegree $(p+2N+1,q+N)$. The kernel of $\partial^*_{\cal
X}$ in this bidegree is covered by the group
\begin{eq}
\llabel{covergroup}
H^{p+2N+1,q+N}(\Sigma^1_s Th_X(V)\wedge {\cal
X}_+,\zz/l)=H^{p+2N,q+N}(Th_X(V)\wedge {\cal
X}_+,\zz/l)
\end{eq}
The image of the projection $pr:Th_X(V)\wedge {\cal X}_+\sr Th_X(V)$ 
in $DM$ is an appropriate twist of the morphism
$$M(X)\oo \zz_{\cal X}\sr M(X)$$
which is an isomorphism by \refdonotneed. Therefore, $pr$
defines an isomorphism on the motivic cohomology with
$\zz/l$-coefficients and we conclude that (\ref{covergroup})
is isomorphic to the group
$$H^{p+2N,q+N}(Th_X(V),\zz/l)=H^{p+2d,q+d}(X,\zz/l)$$
which is zero for $p>q$ by the cohomological dimension theorem.
\end{proof}
\begin{remark}\rm
The end of the proof of Theorem \ref{degree} shows that the first
condition of the theorem can be replaced by the condition that
$\alpha$ does not belong to the image of the homomorphism 
$$H_{-p,-q}(X,\zz/l)\sr H^{p,q}({\cal X},\zz/l).$$
\end{remark}

\subsection{Generalized Rost motives}
\llabel{s4}
In this section we work over fields of characteristic zero to be able
to use the results of Section \ref{comp2} and the motivic duality. All
motives are with $\zz_{(l)}$-coefficients.  We consider $n>0$ and an
embedded smooth simplicial scheme $\cal X$ which satisfies the
following conditions:
\begin{enumerate}
\item There exists a $\nu_n$-variety $X$ such that $M(X)$ lies in
$DM_{\cal X}$
\item There exists an element $\delta$ in $H^{n+1,n}({\cal X},\zz/l)$
such that
\begin{eq}
\llabel{nonzeroeq}
Q_0Q_1\dots Q_{n}(\delta)\ne 0
\end{eq}
where $Q_i$ are the Milnor operations introduced in \cite[Sec.13]{Redpub}.
\end{enumerate}
Under these conditions we will show that there exists a Tate motive
$M_{l-1}$ in $DM_{\cal X}$ which is a direct summand of $M(X)$. Using
the construction of $M_{l-1}$ we will show among other things that 
$$M({\cal X})=M(\check{C}(X)).$$
\begin{remark}\rm
Note that our assumptions imply in particular that $X$ has no
zero cycles of degree prime to $l$.
\end{remark}
\begin{remark}\rm
Modulo the Bloch-Kato conjecture in weight $\le n$ and Conjecture
\ref{lessiseq} (or assuming that for all $i\le n$ there exist a
$\nu_i$-variety $X_i$ such that $M(X_i)$ is in $DM_{\cal X}$),
the condition (\ref{nonzeroeq}) is equivalent to the condition $\delta\ne
0$ (see the proof of Lemma \ref{step2}).
\end{remark}
\begin{remark}\rm
Let ${\cal X}_0$ be the zero term of $\cal X$. Then, modulo the
Bloch-Kato conjecture in weight $\le n$ one has
$$H^{n+1,n}({\cal X},\zz/l)=$$
$$=\bigcap_{\alpha}ker(H^{n+1}_{et}(k,\mu_l^{\oo n})\sr
H^{n+1}_{et}(k(X_{\alpha}),\mu_l^{\oo n}))$$
where $X_{\alpha}$ are the connected components of ${\cal X}_0$ (see
the proof of Lemma \ref{comp1}). Therefore, our conditions on $\cal X$
can be reformulate by saying that there exist $\nu_{i}$-varieties in
$DM_{\cal X}$ for all $i\le n$ and 
$$ker(H^{n+1}_{et}(k,\mu_l^{\oo n})\sr
H^{n+1}_{et}(k(X_{\alpha}),\mu_l^{\oo n}))\ne 0$$
i.e. ${\cal X}_0$ splits a non-zero element in
$H^{n+1}_{et}(k,\mu_l^{\oo n})$.
\end{remark}
\begin{remark}\rm
Extending the previous remark we see that if $k$ contains a primitive
$l$-th root of unity (such that $\mu_l\cong \zz/l$) the results of
this section are applicable to all non-zero elements in
$H^{n+1,n+1}(k,\zz/l)$ which can be split by a
$\nu_n$-variety. Theorem \ref{Rost2} shows that any pure symbol
i.e. the product of $n+1$ elements from $H^{1,1}$ is such an element.
It seems natural to conjecture that the inverse implication also holds
i.e. that an element in $H^{n+1,n+1}(k,\zz/l)$ which can be split by a
$\nu_n$-variety is a pure symbol.
\end{remark}
Set
\begin{eq}
\llabel{defmu}
\mu=\tilde{Q_0}Q_1\dots Q_{n-1}(\delta)
\end{eq}
where $\tilde{Q_0}$ is the intergal-valued Bockstein
homomorphism 
$$H^{*,*}(-,\zz/l)\sr H^{*+1,*}(-,\zz)$$
Then 
$$\mu\in H^{2b+1,b}({\cal X},\zz)$$
where $b=(l^n-1)/(l-1)$. 

Consider $\mu$ as a
morphism in the category of Tate motives over ${\cal X}$ and define
$M=M_{\mu}$ by the distinguished triangle in $DM_{\cal X}$
of the form
\begin{eq}
\llabel{deftr}
\zz_{\cal X}(b)[2b]\stackrel{x}{\sr}
M \stackrel{y}{\sr} \zz_{\cal X}\stackrel{\mu}{\sr}
\zz_{\cal X}(b)[2b+1]
\end{eq}
For any $i<l$ let
\begin{eq}
\llabel{newline}
M_i=S^{i}M
\end{eq}
be the $i$-th symmetric power of $M$. The motive $M_{l-1}$ is called
the {\em generalized Rost motive} defined by $X$ and $\delta$. Note
that $\mu$ is an l-torsion element and therefore we have
$$M_{i}\oo {\bf Q}=\oplus_{j=0}^{i} \qq(jb)[2jb].$$
With integral coefficients $M_i$ does not split into a direct
sum. Instead the distinguished triangles of the form (\ref{seq1}) and
(\ref{seq2}) give us distingished triangles
\begin{eq}
\llabel{seq1n}
M_{i-1}(b)[2b]\sr M_{i}\sr \zz_{\cal X}\sr M_{i-1}(b)[2b+1]
\end{eq}
and
\begin{eq}
\llabel{seq2n}
\zz_{\cal X}(bi)[2bi] \sr M_i\sr
M_{i-1}\sr \zz_{\cal X}(bi)[2bi+1]
\end{eq}
which describe $M_i$ in terms of Tate motives
$$\zz_{\cal X}(jb)[2jb]=M({\cal X})(jb)[2jb]$$
over $\cal X$.  The main goal of this section is to show that
$M_{l-1}$ is a pure motive which is essentially self-dual and which
splits as a direct summand from $M(X)$. It can be shown that this
property is special to $M_{l-1}$ and does not hold for $M_i$ where
$i<l-1$.
\begin{example}\rm\llabel{exl2}
For $l=2$ the Pfister quadric $Q_{\uu{a}}$ defined by
a sequence of invertible elements $(a_1,\dots,a_{n+1})$ of $k$ is a
$\nu_n$-variety. There is a unique non-zero class $\delta$ in
$H^{n+1,n}(\check{C}(Q_{\uu{a}}),\zz/2)$ and it satisfies the
condition (\ref{nonzeroeq}). The corresponding motive $M_1=M$ is
the standard Rost motive considered in \cite{MCpub}.
\end{example}
\begin{example}\rm\llabel{exn0}
Everywhere below we consider the case $n>0$. The case $n=0$ gives a
good motivating example but the construction of $M$ has to be modified
slightly since (\ref{defmu}) clearly makes no sense in this case. A
$\nu_0$-variety is a variety of dimension zero and degree non divisble
by $l^2$. The simpliest interesting example is $X=Spec(E)$
where $E$ is an extension of degree $l$. In order to have
$H^{1,0}(\check{C}(X),\zz/l)\ne 0$, $k$ must contain a primitive
$l$-th root of unity. In that case we may set $\mu=\delta$ and define
$M$ as a motive with $\zz/l$-coefficients given by 
$$\zz/l\sr M\sr \zz/l\stackrel{\delta}{\sr}\zz/l[1]$$
over $\check{C}(X)$. Then $M_{l-1}$ is the motive of $Spec(E)$ with
$\zz/l$-coefficients.
\end{example}
We start with several results about the motives $M_i$ which do not depend
on any subtle properties of $X$ or $\mu$. For the proof of these
results it will be convenient to consider our motives as relative Tate
motives over $\cal X$.
\begin{lemma}
\llabel{M1}
For any $i=1,\dots,l-1$ there exists a morphism 
$$e_i:M_i\oo M_i\sr
\zz_{\cal X}(bi)[2bi]$$
such that $(M_i,e_i)$ is an internal Hom-object from $M_i$ to $\zz_{\cal
X}(bi)[2bi]$ in $DM_{\cal X}$. 
\end{lemma}
\begin{proof}
Consider first the
case $i=1$. Since the Tate objects are quasi-invertible there exist
internal Hom-objects $(\zz_{\cal X},u)$ (resp. $(\zz_{\cal
X}(b)[2b],v)$) from $\zz_{\cal X}(b)[2b]$ (resp. $\zz_{\cal X}$) to
$\zz_{\cal X}(b)[2b]$. The dual $D\mu$ is again $\mu$ and applying
\refappmain to the distinguished triangle defining $M$ we
conclude that there exists $e_1$ with the required property.

We can now define $e_i$ for $i>1$ as the morphism 
$$e_i:M_i\oo M_i\cong S^{i}(M\oo
M)\stackrel{S^{i}e_1}{\sr}  S^{i}(\zz_{\cal X}(b)[2b])=\zz_{\cal
X}(bi)[2bi]$$
\refdualten implies immediately that $(M_i,e_i)$ is an
internal Hom-object from $M_i$ to $\zz_{\cal X}(bi)[2bi]$.
\end{proof}
Consider the homomorphism
\begin{eq}
\llabel{endoeq}
End(M_i)\sr \oplus_{j=0}^i \zz
\end{eq}
defined by the slice functor over $\cal X$ and the identifications
$$End(s_{bj}(M_i))=End(\zz)=\zz,\,\,\, j=0,\dots,i$$
\begin{lemma}
\llabel{endo} The image of (\ref{endoeq}) is contained in the subgroup
of elements $(c_0,\dots,c_i)$ such that $c_k=c_j\,mod\,l$ for all
$k,j$.
\end{lemma}
\begin{proof}
Let $w$ be an endomorphism of $M_i$, $c_j$ be the $j$-th slice of $w$
and $c_{j+1}$ the $(j+1)$-st slice of $w$. We need to show that
$c_j=c_{j+1}\,mod\,l$. Consider the object $\Pi{\ge jb}\Pi_{<
jb+2}(M_i)$. By Lemma \ref{mainseq} we have
$$\Pi_{\ge jb}(M_i)=M_j(jb)[2jb]$$
$$\Pi_{< jb+2}(M_j(jb)[2jb])=(\Pi_{<2}(M_j))(jb)[2jb]=M(jb)[2jb]$$
This reduces the problem to the case $j=0$ and $i=1$ i.e. to an
endomorphism
$${M}\sr {M}.$$
Since the defining triangle for $M$ coincides with one of the
triangles of the slice tower of $M$ it is natural in $M$. This fact
together with the fact that $\mu$
is non-zero modulo $l$ implies the result we need.
\end{proof}
\begin{remark} \rm It is easy to see that the image of (\ref{endoeq})
in fact coincides with the subgroup of Lemma \ref{endo}.
\end{remark}
\begin{cor}
\llabel{simple} Let $w:M_i\sr M_i$ be a morphism such
that the square
$$
\begin{CD}
M_i @>w>> M_i\\
@VVV @VVV\\
\zz_{\cal X} @>c>> \zz_{\cal X}
\end{CD}
$$
commutes for an integer $c$ prime to $l$. Then $w$ is an isomorphism.
\end{cor}
\begin{proof}
Since the slice functor is conservative on Tate motives it is
sufficient to show that $w$ is an isomorphism on each slice. Our
assumption implies that $w$ is $c$ on the zero slice and since $c$ is
prime to $l$ it is an isomorphism there. We conlude that $w$ is also
prime to $l$ and hence an isomorphism on the other slices by Lemma
\ref{endo}.
\end{proof}
Let 
$$\pi_{\cal X}:M(X)\sr \zz_{\cal X}$$
be the unique morphism such that the composition
$$M(X)\sr \zz_{\cal X}\sr \zz$$
is the structure morphism $\pi:M(X)\sr \zz$.
\begin{lemma}
\llabel{lex} For any smooth $X$ such that $M(X)$ is in $DM_{\cal X}$
there exists $\lambda$ which makes the diagram
\begin{eq}\llabel{lam}
\begin{CD}
M(X) @>\lambda>> M_i\\
@V\pi_{\cal X}VV @VVS^{i}(y)V\\
\zz_{\cal X} @>Id>> \zz_{\cal X}
\end{CD}
\end{eq}
commutative.
\end{lemma}
\begin{proof}
The distinguished triangle of the form (\ref{seq1n}) for $M_{i}$ 
shows that the obstruction to the existence of $\lambda$ lies in the
group of morphisms
$$Hom(M(X),M_{i-1}(b)[2b+1])$$
Using induction on $i$ and the sequences (\ref{seq2n}) to compute
these group we see that it is build out of the groups
$$Hom(M(X),\zz_{\cal X}(bj)[2bj+1])=H^{2bj+1,bj}(X,\zz)$$
where the equality holds by \refrl. Since $X$ is smooth these
groups are zero.
\end{proof}
Let us now consider the motive $M_i$ for $i=l-1$. To simplify the
notations we set $d=b(l-1)=l^n-1$.
\begin{proposition}
\llabel{nondiv1} For any $\lambda$ which makes the square (\ref{lam})
commutative (for $i=l-1$) the composition
$$\lambda\tau_{\cal X}:\zz_{\cal
X}(d)[2d]\sr M_{l-1}$$
is not divisible by $l$.
\end{proposition}
\begin{proof}
In view of Theorem \ref{degree} it is enough to construct a non-zero
motivic cohomology class $\alpha$ in $H^{p,q}({\cal X},\zz/l)$ for
some $p>q$ such that $\alpha$ vanishes on $M_{l-1}$ and such that
$Q_n(\alpha)=0$. We set $\alpha=Q_n(\mu\, mod\, l)$. Let us verify
that all the required conditions hold. The bidegree of $\alpha$ is
$(2b+2d+2,b+d)=(lb+2,lb)$. In particular the dimension is greater than
weight. By Lemma \ref{margolis} the
n-th motivic Margolis homology of the unreduced suspension
$\tilde{\cal X}$ of $\cal X$ is zero. Hence if $Q_n(\mu)=0$ then
$\mu=Q_n(\gamma)$ where
$$\gamma\in H^{2b-2d+1,b-d}(\tilde{\cal X},\zz/l)$$
For $l>2$ and $n>0$ we have $b-d<0$ and this group is zero. For $l=2$
we have $b=d$ and the group $H^{1,0}(\tilde{\cal X},\zz/2)$ is zero
from the long exact sequence relating the motivic cohomology of
$\tilde{\cal X}$ and the motivic cohomology of $\cal X$. Since $\mu\ne
0$ by our assumption (\ref{nonzeroeq}) we conclude that $\alpha\ne 0$.

The condition $Q_n(\alpha)=0$ follows immediately from the fact that
$Q_n^2=0$ (see \cite[Prop. 13.3, 13.4]{Redpub}). It remains to check that
$\alpha$ vanishes on $M_{l-1}$. In view of Theorem \ref{maincomp} and
the definition of the operation $\phi_{l-1}$ the class $\beta
P^{b}(\mu)$ vanishes on $M_{l-1}$. Since $Q_i(\mu)=0$ for $i<n$ we
conclude by Lemma \ref{scomp} that
$$Q_n(\mu)=\beta P^b(\mu)$$
which finishes the proof of the proposition for $l>2$. The proof for
$l=2$ can be easily deduced from the results of \cite{MCpub}.
%
\end{proof}
\begin{lemma}
\llabel{scomp}
One has the following equality in the motivic Steenrod algebra for
$l>2$:
\begin{eq}
\llabel{srel}
Q_0P^{b}=P^bQ_0+P^{b-1}Q_1+P^{b-l-1}Q_2+\dots + P^0Q_n
\end{eq}
\end{lemma}
\begin{proof}
Since $l>2$ the subalgebra of the motivic Steenrod algebra generated
by operations $\beta, P^i$ is isomorphic to the usual topological
Steenrod algebra. In the topological Steenrod algebra the equation 
follows by easy induction on $n$ from the commutation relation for the
Milnor basis given in \cite[Theorem 4a]{Milnor3}.
\end{proof}
Let $\Delta^*:M(X)\oo M(X)\sr \zz(d)[2d]$ be the morphism defined by
the diagonal and
$$e_X=\Delta^*_{\cal X}:M(X)\oo M(X)\sr \zz_{\cal
X}(d)[2d]$$
the morphism which corresponds to $\Delta^*$ by \refrl. 
\begin{proposition}
\llabel{dualx} The pair $(M(X),e_X)$ is an internal
Hom-object from $M(X)$ to $\zz_{\cal X}(d)[2d]$ in $DM_{\cal X}$.
\end{proposition}
\begin{proof}
It follows from \refmotdual and \refintres.
\end{proof}
Define $D\lambda$ as the dual of $\lambda$ with respect to $e_X$ and
$e_M$. 
\begin{lemma}
\llabel{over}
There exists $c$ prime to $l$ such that the diagram 
$$
\begin{CD}
M_{l-1} @>\lambda D\lambda>> M_{l-1}\\
@VVV @VVV\\
\zz_{\cal X} @>c>> \zz_{\cal X}
\end{CD}
$$
commutes. In particular, $\lambda$ is a split epimorphism.
\end{lemma}
\begin{proof}
We will show that there is $c$ such that the diagram
\begin{eq}\llabel{twosquares}
\begin{CD}
M_{l-1} @>D\lambda>> M(X) @>\lambda>> M_{l-1}\\
@VVS^{l-1}(y)V @VV\pi_{\cal X}V @VVS^{l-1}(y)V\\
\zz_{\cal X} @>c>> \zz_{\cal X} @>Id>> \zz_{\cal X}
\end{CD}
\end{eq}
commutes. Since the right hand side square commutes by definition of
$\lambda$ we only have to consider the left hand side square. Onserve
first that 
$$\pi_{\cal X}=D\tau_{\cal X}.$$
On the other hand
$$S^{l-1}(y)=DS^{l-1}(x)$$
Using the fact that $D(gf)=D(f)D(g)$ we see that to show that the left
hand side square commutes it is enough to show that there is $c$ prime
to $l$ such that the square
$$
\begin{CD}
\zz_{\cal X}(d)[2d] @>c>> \zz_{\cal X}(d)[2d]\\
@V\tau_{\cal X}VV @VVS^{l-1}(x)V\\
M(X)@>\lambda>> M_{l-1}
\end{CD}
$$
commutes. The fact that there exists $c\in\zz$ which makes this
diagram commutative follows immediately from the distinguished
triangles (\ref{seq1}) and the fact that Tate objects of higher weight
admit no nontrivial morphisms to Tate objects of lower weight. The
fact that $c$ must be prime to $l$ follows from Proposition
\ref{nondiv1}.
\end{proof}
Combining Lemma \ref{over} with Corollary \ref{simple} we conclude
that $\lambda D\lambda$ is an isomorphism. Let $\phi$ be its
inverse. Then the composition 
$$p:D\lambda\circ \phi\circ \lambda:M(X)\sr
M(X)$$
is a projector i.e. $p^2=p$ and its image is $M_{l-1}$. We conclude
that $M_{l-1}$ is a direct summand of $M(X)$. Together with 
\refspcase this implies the following important result.
\begin{theorem}
\llabel{mrestr}
The motive $M_{l-1}$ is restricted.
\end{theorem}
Combining Theorem \ref{mrestr} with Lemmas \ref{M1} and \refintres
we get the following duality theorem for $M_{l-1}$.
\begin{cor}
\llabel{mdual}
Let $e_M'$ be the composition
$$M_{l-1}\oo M_{l-1}\stackrel{e_M}{\sr} \zz_{\cal X}(d)[2d]\sr
\zz(d)[2d]$$
Then $(M_{l-1},e'_M)$ is an internal Hom-object from $M_{l-1}$ to
$\zz(d)[2d]$ in the category $DM^{eff}_{-}(k)$.
\end{cor}
\begin{proposition}
\llabel{universal}
Under the assumptions of this section one has 
$$M({\cal X})\cong M(\check{C}(X))$$
where the motives are considered with $\zz_{(l)}$-coefficients.
\end{proposition}
\begin{proof}
By \refwheneq it is sufficient to show that for any smooth $Y$ in
$DM_{\cal X}$ there exists a morphism $M(Y)\sr M(X)$ over
$\zz$. Diagram (\ref{twosquares}) shows that $c^{-1}D\lambda$ is a
morphism $M_{l-1}\sr M(X)$ over $\zz$. On the other hand Lemma
\ref{lex} shows that there
is a morphism $M(Y)\sr M_{l-1}$ over $\zz$. The statement of the
proposition follows.
\end{proof}

\subsection{The Bloch-Kato conjecture}
\llabel{s5}
In this section we use the techniques developed above to prove the
following theorem.
\begin{theorem}
\llabel{BK0}
Let $k$ be a field of characteristic zero which contains a primitive
$l$-th root of unity. Then the norm residue homomorphisms
$$K_n^M(k)/l\sr H^n_{et}(k,\mu_l^{\oo n})$$
are isomorphisms for all $n$.
\end{theorem}
In the next section we will extend this theorem to all fields of 
characteristic not equal to $l$. The statement of Theorem \ref{BK0} is
know as the Bloch-Kato conjecture (see \cite{MCpub}).

As was shown in \cite[pp.96-97]{MCpub}, in order to prove Theorem \ref{BK0} it is
sufficient to construct for any $k$ of characteristic zero and any
sequence of invertible elements $\uu{a}=(a_1,\dots,a_n)$ of $k$, a
field extension $K_{\uu{a}}$ of $k$ such that the following two
conditions hold:
\begin{enumerate}
\item the image of $\uu{a}$ in $K_n^M(K_{\uu{a}})$ is divisible by
$l$,
\item the homomorphism of the Lichtenbaum (etale) motivic cohomology
groups
$$H^{n+1,n}_{et}(K,\zz_{(l)})\sr
H^{n+1,n}_{et}(K_{\uu{a}},\zz_{(l)})$$
is a monomorphism.
\end{enumerate}
We say that a smooth connected scheme $X$ splits $\uu{a}$ modulo $l$
if $\uu{a}$ becomes zero in $K_n^M(k(X))/l$ where $k(X)$ is the
function field of $X$. We use
the notation $H_{-1,-1}(X,\zz)$ for the motivic homology group
$$H_{-1,-1}(X,\zz)=Hom_{DM}(\zz,M(X)(1)[1])$$
For $X=Spec(k)$ this group is $k^*$ and for a general $X$ it has a
description in terms of cycles with coefficients in $K^M_*$. If $X$ is
smooth projective of dimension $d$ over a field of characteristic zero
then the motivic duality theorem implies that
$$H_{-1,-1}(X,\zz)=H^{2d+1,d+1}(X,\zz)$$
\begin{definition}
\llabel{lenu} A smooth projective variety $X$ over $k$ is called a
$\nu_{\le n}$-variety if $X$ is a $\nu_n$-variety and for all $i< n$
there exists a $\nu_i$-variety $X_i$ and a morphism $X_i\sr X$.
\end{definition}
It seems likely that the following conjecture holds.
\begin{conjecture}
\llabel{lessiseq}
Any $\nu_n$-variety is a $\nu_{\le n}$-variety.
\end{conjecture}
A key point in our proof of Theorem \ref{BK0}
is the following result announced by Markus Rost and proved in \cite{SJ}.
\begin{theorem}
\llabel{Rost2} For any $\uu{a}=(a_1,\dots,a_n)$ there exists a
$\nu_{\le (n-1)}$-variety $X$ such that:
\begin{enumerate}
\item $X$ splits $\uu{a}$
\item the sequence
$$H_{-1,-1}(X\times X,\zz)\stackrel{pr_1-pr_2}{\longrightarrow}
H_{-1,-1}(X,\zz)\sr k^*$$
is exact.
\end{enumerate}
\end{theorem}
In order to prove Theorem \ref{BK0} we will show that for any $X$
satisfying the conditions of Theorem \ref{Rost2} the homomorphism
$$H^{n+1,n}_{et}(k,\zz_{(l)})\sr H^{n+1,n}_{et}(k(X),\zz_{(l)})$$
is injective. We will have to assume during the proof that Theorem
\ref{BK0} holds in degrees $\le (n-1)$.
\begin{lemma}
\llabel{lastlemma} Assume that Theorem \ref{BK0} holds in degrees $\le
n-1$ and the $\uu{a}=(a_1,\dots,a_n)$ is a symbol which is not zero in
$K_n^M(k)/l$. Then the image of $\uu{a}$ in $H^n_{et}(k,\mu_l^{\oo
n})$ is not zero.
\end{lemma}
\begin{proof}
By standard transfer argument it is enough to prove the lemma for
fields $k$ which have no extensions of degree prime to $l$. In
particular $\mu_l\cong \zz/l$.  We proceed by induction on $n$. We
know the statement for $n=1$. Let 
$$E=k[t]/(t^l=a_n)$$
be the cyclic extension of
degree $l$ corresponding to $a_n$ and $\alpha$ the class in
$H^{1}_{et}$ corresponding to $a_n$. Let $\gamma$ be the image of
$(a_1,\dots,a_{n-1})$ in $H^{n-1}_{et}$. By induction we may assume
that $\gamma\ne 0$. By \cite[Proposition 5.2]{MCpub} we have an exact
sequence
$$H^{n-1}_{et}(E,\zz/l)\stackrel{N_{E/k}}{\sr}
H^{n-1}_{et}(k,\zz/l)\stackrel{\alpha}{\sr}
H^{n}_{et}(k,\zz/l)\sr H^{n}_{et}(E,\zz/l)$$
and therefore if $\gamma\alpha=0$ then $\gamma=N_{E/k}(\gamma')$. In
the weight $n-1$ etale cohomology are isomorphic to the Milnor
K-theory by our assumption. Therefore $(a_1,\dots,a_{n-1})$ is the
norm of an element in $K_{n-1}^M(E)$ and we conclude that 
$$(a_1,\dots,a_{n-1},a_n)=(a_1,\dots,a_{n-1})\wedge (a_n)=0$$
\end{proof}
\begin{lemma}
\llabel{comp1} Assume that Theorem \ref{BK0} holds in degrees $\le
(n-1)$, $\uu{a}$ is not zero in $K_n^M(k)/l$ and $X$ is a disjoint
union of smooth schemes such that each component of $X$ splits
$\uu{a}$. Then there exists a non-zero element $\delta$ in
$H^{n,n-1}(\check{C}(X),\zz/l)$.
\end{lemma}
\begin{proof}
Since we assumed the Bloch-Kato conjecture in weight $\le (n-1)$ we
know by \cite{MCpub} that 
$$H^{*,n-1}(-,\zz/l)={\bf H}^{*}_{Nis}(-,B/l(n-1))$$
where $B/l(n-1)$ is the truncation $\tau^{\le(n-1)}$ of the total
direct image of the sheaf $\mu_l^{\oo(n-1)}$ from the etale to the
Nisnevich topology. In particular for any ${\cal X}$ one has
$$H^{n,n-1}({\cal X},\zz/l)=ker(H^n_{et}({\cal X},\mu_l^{\oo(n-1)})\sr
H^0({\cal X}, \uu{H}^n_{et}({\cal X},\mu_l^{\oo(n-1)})))$$
where $\uu{H}^n_{et}$ is the Nisnevich sheaf associated with the
presheaf $H^{n}_{et}$. For a simplicial scheme $\cal X$ and any
sheaf $F$ we have $H^0({\cal X},F)\subset H^0({\cal X}_0,F)$ where
${\cal X}_0$ is the sero term of $\cal X$. If ${\cal X}_0$ is a
disjoint union of smooth schemes and $F$ is a homotopy invariant
Nisnevich sheaf with transfers we further have
$$H^0({\cal X}_0,F)\subset \prod_{\alpha} H^0(Spec(k(X_{\alpha})),F)$$
where $X_{\alpha}$ are the connected components of ${\cal X}_0$.
Therefore for ${\cal X}=\check{C}(X)$ we get
$$H^{n,n-1}({\cal X},\zz/l)=ker(H^n_{et}({\cal X},\mu_l^{\oo(n-1)})\sr
\prod_{\alpha} H^n_{et}(Spec(k(X_{\alpha})),\mu_l^{\oo(n-1)}))$$
where $X_{\alpha}$ are the connected components of $X$. If $X\ne
\emptyset$ and $F$ is an etale sheaf we have (cf. the proof of \cite[Lemma
7.3]{MCpub})
$$H^{n}_{et}(\check{C},F)=H^n_{et}(Spec(k),F)$$
therefore  
$$H^{n,n-1}({\cal X},\zz/l)=$$
$$=ker(H^n_{et}(Spec(k),\mu_l^{\oo(n-1)})\sr
\prod_{\alpha} H^n_{et}(Spec(k(X_{\alpha})),\mu_l^{\oo(n-1)})).$$
Recall now that we assumed that $k$ contains a primitive $l$-th root
of unity. Therefore we can replace $\mu_l^{\oo(n-1)}$ by $\mu_l^{\oo
n}$ and we conclude that $H^{n,n-1}({\cal X},\zz/l)$ contains 
$$ker(H^n_{et}(Spec(k),\mu_l^{\oo n})\sr
\prod_{\alpha} H^n_{et}(Spec(k(X_{\alpha})),\mu_l^{\oo n}))$$
which is non zero by our condition that each $X_{\alpha}$ splits
$\uu{a}$ and Lemma \ref{lastlemma}.
\end{proof}
Set ${\cal X}=\check{C}(Y)$ where $Y$ is the disjoint union of all (up
to an isomorphism) smooth schemes which split $\uu{a}$ and let
$\tilde{\cal X}$ be the unreduced suspension of ${\cal X}$. Note that
for a smooth connected variety $X$ one has $M(X)\in DM_{\cal X}$ if
and only if $X$ splits $\uu{a}$.
\begin{lemma}
\llabel{forref}
Under the assumption that Theorem \ref{BK0} holds in weights $< n$ one
has
$$\tilde{H}^{p,q}(\tilde{\cal X},\zz/l)=0$$
for all $q\le n-1$ and $p\le q+1$.
\end{lemma}
\begin{proof}
By \cite[Cor. 6.9]{MCpub} and our assumption that Theorem \ref{BK0} holds in
weights $< n$ we conclude that for $q\le n-1$ and $p\le q+1$ we have
$$H^{p,q}(\tilde{\cal X},\zz/l)\subset H^{p,q}_{et}(\tilde{\cal
X},\zz/l).$$
The right hand side group is zero for all $p$ and $q$ by \cite[Lemma 7.3]{MCpub}.
\end{proof}
\begin{lemma}
\llabel{step2} Let $\delta$ be as in Lemma \ref{comp1}. Then
$$Q_{n-1}\dots Q_0(\delta)\ne 0$$
\end{lemma}
\begin{proof}
The cofibration sequence which defines $\tilde{\cal X}$ gives us a
homomorphism $H^{p,q}({\cal X})\sr H^{p+1,q}(\tilde{\cal X})$ which is
a monomorphism for $p>q$. Let $\tilde{\delta}$ be the image of
$\delta$ in $H^{n+1,n-1}(\tilde{\cal X})$. Since $\delta\ne 0$ we have
$\tilde{\delta}\ne 0$. Let us show that
$$Q_i\dots Q_0(\tilde{\delta})\ne 0$$
for all $i<n$. Assume by induction that
$$Q_{i-1}\dots Q_0(\tilde{\delta})\ne 0$$
By Theorem \ref{Rost2} there exists a $\nu_{\le (n-1)}$-variety $X$ which
splits $\uu{a}$. By our construction we have $M(X)\in DM_{\cal
X}$. Therefore by Lemma \ref{margolis} the motivic Margolis homology
$\tilde{MH}^{*,*}_i$ of $\tilde{\cal X}$ are zero for all $i<n$. Hence
$Q_i\dots Q_0(\tilde{\delta})=0$ if and only if there exists $u$ such
that
\begin{eq}
\llabel{kercov}
Q_i(u)=Q_{i-1}\dots Q_0(\tilde{\delta})
\end{eq}
Let us make some degree computations which will also be useful
below. The composition $Q_{i-1}\dots Q_0$ shifts dimension by
$$1+2l-1+\dots + 2l^{i-1}-1=-i+2l(l^{i-1}-1)/(l-1)+2$$
and weight by
$$0+l-1+\dots+l^i-1=-i+l(l^{i-1}-1)/(l-1)+1$$
Therefore the kernel of $Q_i$ on $Q_{i-1}\dots
Q_0(\tilde{H}^{p,q}(-,-))$ is covered by the group of dimension
$$-i+2l(l^{i-1}-1)/(l-1)+2-2l^i+1=-i+2lw+3$$
and weight
$$-i+l(l^{i-1}-1)/(l-1)+1-l^i+1=-i+lw+2$$
where $w=(l^{i-1}-1)/(l-1)-l^i$. Note that $w\le -1$ and $lw\le
-2$. Therefore the bidegree of $u$ in (\ref{kercov}) is
$(n+1-i+2lw+3, n-1-i+lw+2)$. We conclude that the weight of $u$ is
$\le n-1$ and the difference between the dimension and the weight is
$$n+1-i+2lw+3-(n-1-i+lw+2)=3+lw\le 1$$
By Lemma \ref{forref} we conclude that $u=0$ which contradicts our
inductive assumption that $Q_{i-1}\dots Q_0(\tilde{\delta})\ne 0$.
\end{proof}
Define $\mu$ as in (\ref{defmu})
starting with $\delta$ and let $M_i$ be the motive defined
by (\ref{newline}). In view of Lemma \ref{step2} the results of
the previous section are applicable. In particular Proposition
\ref{universal} implies the following.
\begin{lemma}
\llabel{isuniversal}
Let $X$ be a $\nu_{n-1}$-variety which splits $\uu{a}$. Then 
$$M({\cal X})=M(\check{C}(X)).$$
\end{lemma}
\begin{lemma}
\llabel{step3} Let $X$ be a $\nu_{n-1}$-variety which splits
$\uu{a}$. Then there is an exact sequence
$$H^{n+1,n}({\cal X},\zz_{(l)})\sr H^{n+1,n}_{et}(k,\zz_{(l)})\sr
H^{n+1,n}_{et}(k(X),\zz_{(l)})$$
\end{lemma}
\begin{proof}
The morphism $Spec(k(X))\sr Spec(k)$ admits a decomposition
$$Spec(k(X))\sr X\sr {\cal X}\sr Spec(k)$$
where the middle arrow is the natural morphism from $X$ to
$\cal X$. By \cite[Lemma 7.3]{MCpub} the last arrow defines an isomorphism on
$H^{n+1,n}_{et}(-,\zz_{(l)})$. Therefore it is sufficient to show that
the sequence
$$H^{n+1,n}({\cal X},\zz_{(l)})\sr H^{n+1,n}_{et}({\cal X},\zz_{(l)})\sr
H^{n+1,n}_{et}(k(X),\zz_{(l)})$$
is exact. The composition of two morphisms is zero because it factors
through 
$$H^{n+1,n}(k(X),\zz_{(l)})=0$$
Let $\zz_{(l)}^{et}(n)$ be the object in $DM^{eff}_{-}(k)$ which
represents the etale motivic cohomology of weight $n$ and let $L(n)$
be its canonical truncation at the level $n+1$ (see
\cite[p.90]{MCpub}). Consider a distinguished triangle of the form
$$\zz_{(l)}(n)\sr L(n)\sr K(n)\sr
\zz_{(l)}(n)[1]$$
where the first arrow corresponds to the natural morphism
$$\zz_{(l)}(n)\sr \zz_{(l)}^{et}(n).$$

Let $x$ be an element in 
$$H^{n+1,n}_{et}({\cal X},\zz_{(l)})={\bf H}^{n+1}({\cal X},L(n))$$
which goes to zero in 
$$H^{n+1,n}_{et}(k(X),\zz_{(l)})={\bf
H}^{n+1}(k(X),L(n)).$$
We have to show that the image $x'$ of $x$ in ${\bf H}^{n+1}({\cal
X},K(n))$ is zero.  By \cite[Lemma 6.13]{MCpub} $x'$
maps to zero in ${\bf H}^{n+1}(X,K(n))$. By Lemma \ref{lex} we know
that the
morphism from $M(X)$ to $M(\cal X)$ factors as
\begin{eq}\llabel{factoreq}
M(X_{\uu{a}})\stackrel{\lambda}{\sr} M_{l-1}\sr M({\cal X})
\end{eq}
where the first arrow is a split epimorphism by Lemma \ref{over}. By
\cite[Lemma 6.7]{MCpub}, $L(n)$ and $K(n)$ are complexes of sheaves
with transfers with homotopy invariant cohomology sheaves.  Therefore
$Hom_{DM}(M_{l-1},K(n)[n+1])$ is defined and (\ref{factoreq}) shows
that the image of $x'$ in $Hom_{DM}(M_{l-1},K(n)[n+1])$ is zero. We
conclude that $x'=0$ from (\ref{seq1n}) and the following lemma.
\begin{lemma}
\llabel{need7}
$Hom_{DM}(M_{l-2}(b)[2b],K(n)[n+1])=0$.
\end{lemma}
\begin{proof}
Using the distinguished triangles for $M_i$ it is sufficient to show
that
$$Hom_{DM}(M({\cal X})(q)[2q], K(n)[n+1])=0$$
for all $q>0$. This is an immediate corollary of \cite[Lemma
6.13]{MCpub} and our assumption that Theorem \ref{BK0} holds in weights
$< n$.
\end{proof}
\end{proof}
In view of Lemma \ref{step3} in order to finish the proof of Theorem
\ref{BK0} it remains to prove the following result.
\begin{proposition}
\llabel{realmain}
$H^{n+1,n}({\cal X},\zz_{(l)})=0$
\end{proposition}
The proof is given in Lemmas \ref{step6}-\ref{step4} below.
\begin{lemma}
\llabel{step6}
There is a monomorphism
\begin{eq}
\llabel{shouldbemono}
H^{n+1,n}({\cal X},\zz_{(l)})\sr H^{2lb+2,lb+1}({\cal
X},\zz_{(l)})
\end{eq}
\end{lemma}
\begin{proof}
The cofibration sequence which defines $\tilde{\cal X}$ implies that
it is enough to show that there is a monomorphism
$$H^{n+2,n}(\tilde{\cal X},\zz_{(l)})\sr H^{2lb+3,lb+1}({\cal
X},\zz_{(l)})$$
Let $X$ be a $\nu_{\le (n-1)}$ variety which splits $\uu{a}$. Since
$X$ is a $\nu_{\le 0}$-variety it has a point over a finite field
extension of degree not divisible by $l^2$. Therefore, the motivic
cohomology of $\tilde{\cal X}$ are of exponent $l$ by
\cite[Lemma 9.3]{MCpub}. Therefore the projection from the motivic cohomology
with the $\zz_{(l)}$ coefficients to the motivic cohomology with the
$\zz/l$ coefficients is injective. Therefore it is sufficient to show
that there is a monomorphism
\begin{eq}
\llabel{zlmono}
H^{n+2,n}(\tilde{\cal X},\zz/l)\sr H^{2lb+3,lb+1}({\cal
X},\zz/l)
\end{eq}
which takes the images of the integral classes to the images of the
integral classes. Consider the composition of cohomological operations
\begin{eq}
\llabel{isinj} Q_{i}\dots Q_{1}:H^{n+2,n}(\tilde{\cal
X}_{\uu{a}},\zz/l)\sr
H^{2l(l^i-1)/(l-1)+n+2-i,l(l^i-1)/(l-1)+n-i}(\tilde{\cal
X}_{\uu{a}},\zz/l)
\end{eq}
For $i=n-1$ it is of the form (\ref{zlmono}) and we know by
\cite[Lemma 7.2]{MCpub} that $Q_i$ take the images of integral classes
to the images of integral classes. Let us show that it is a mono for
all $i\le n-1$.  By Lemma \ref{margolis} we know that the motivic
Margolis homology of $\tilde{\cal X}$ are zero. The computations made
in the proof of Lemma \ref{step2} show that the kernel of $Q_i$ on
$Q_{i-1}\dots Q_1(H^{n+2,n})$ is covered by the group of bidegree
$(p,q)$ where
$$p=4+2lw+n-i$$
$$q=2+lw+n-i$$
$$w=(l^{i-1}-1)/(l-1)-l^{i-1}.$$
We have $w\le -1$ and therefore $q\le n-i$ and $p\le q$. We conclude
that the covering group is zero by Lemma \ref{forref}.
\end{proof}
\begin{lemma}
\llabel{ls1}
There is an epimorphism
$$ker(H^{2b(l-1)+1, b(l-1)+1}(M_{l-1},\zz_{(l)})\sr H^{1,1}({\cal
X},\zz_{(l)}))\sr H^{2lb+2,lb+1}({\cal
X},\zz_{(l)})$$
\end{lemma}
\begin{proof}
Let $X$ be a $\nu_{n-1}$-variety which splits $\uu{a}$. Consider the
sequences (\ref{seq1n}) and (\ref{seq2n}) for $i=l-1$.  By Lemma
\ref{over}, the motivic cohomology of $M_{l-1}$ embed into the motivic
cohomology of $X$ and in particular vanish where the motivic
cohomology of $X$ vanish.

From the first sequence and the fact that $lb+1>(l-1)b=dim(X)$ we
conclude that there is an epimorphism
\begin{eq}
\llabel{group2}
H^{2b(l-1)+1, b(l-1)+1}(M_{l-2},\zz_{(l)})\sr H^{2lb+2,lb+1}({\cal
X},\zz_{(l)})
\end{eq}
From the second sequence and the fact that $H^{0,1}({\cal
X},\zz_{(l)})=0$ we conclude that the left hand side of (\ref{group2})
is the kernel of the homomorphism
$$H^{2b(l-1)+1, b(l-1)+1}(M_{l-1},\zz_{(l)})\sr H^{1,1}({\cal
X},\zz_{(l)}).$$
\end{proof}
\begin{lemma}
\llabel{dualstep}
One has:
$$ker(H^{2b(l-1)+1, b(l-1)+1}(M_{l-1},\zz_{(l)})\sr H^{1,1}({\cal
X},\zz_{(l)}))=$$
$$=ker(Hom(\zz_{(l)}, M_{l-1}(1)[1])\sr Hom(\zz_{(l)},
\zz_{(l)}(1)[1]))$$
\end{lemma}
\begin{proof}
Since the motivic cohomology in the bidegree $(1,1)$ in the Zariski
and the etale topologies coincide and the etale motivic cohomology of
$\cal X$ coincide with the etale motivic cohomology of the point we
have
$$H^{1,1}({\cal
X},\zz_{(l)})=H^{1,1}(Spec(k),\zz_{(l)})$$
By duality estabilished in Corollary \ref{mdual} we have
$$H^{2b(l-1)+1, b(l-1)+1}(M_{l-1},\zz_{(l)})=Hom(\zz_{(l)},
M_{l-1}(1)[1])$$
and one verifies easily that the dual of the morphism
$$\tau_M:\zz(d)[2d]\sr M_{l-1}$$
is the morphism $\pi_M:M_{l-1}\sr \zz$. The statement of the lemma
follows.
\end{proof}
\begin{lemma}
\llabel{step4}
The
homomorphism
$$Hom(\zz_{(l)}, M_{l-1}(1)[1])\sr Hom(\zz_{(l)},
\zz_{(l)}(1)[1])$$
is a monomorphism.
\end{lemma}
\begin{proof}
The distinguished triangle (\ref{seq1n}) together with the obviuous
fact that 
$$Hom(\zz, M({\cal X}(bj)[2bj]))=0$$
for $j>0$, implies that the homomorphism
$$Hom(\zz_{(l)}, M_{l-1}(1)[1])\sr Hom(\zz_{(l)},
M({\cal X})(1)[1])$$
is a monomorphism. It remains to see that
\begin{eq}
\llabel{lasteq}
Hom(\zz_{(l)},
M({\cal X})(1)[1])\sr Hom(\zz,\zz(1)[1])=k^*
\end{eq}
is a monomorphism. By Lemma \ref{isuniversal} we may assume that
${\cal X}=\check{C}(X)$ where $X$ is a smooth variety satisfying the
conditions of Theorem \ref{Rost2}. The spectral sequence which starts
from motivic homology of $X$ and converges to the motivic
homology of $\cal X$ shows that
$$Hom(\zz_{(l)},
M({\cal
X})(1)[1])=coker(H_{-1,-1}(X^2,\zz)\stackrel{pr_1-pr_2}{\sr}H_{-1,-1}(X,\zz))$$
We conclude that (\ref{lasteq}) is a mono by Theorem \ref{Rost2}.
\end{proof}
The deduction of the following two results from Theorem \ref{BK0} can be found in \cite{MCpub}. 
\begin{theorem}
\llabel{BK1} Let $k$ be a field of characteristic $\ne l$. Then the
norm residue homomorphisms
$$K_n^M(k)/l\sr H^n_{et}(k,\mu_l^{\oo n})$$
are isomorphisms for all $n$.
\end{theorem}
\begin{theorem}
\llabel{BK2} Let $k$ be a field and $\cal X$ a pointed smooth simplicial
scheme over $k$. Then one has:
\begin{enumerate}
\item for any $n>0$ the homomorphisms
$$\tilde{H}^{p,q}({\cal X},\zz/n)\sr \tilde{H}^{p,q}_{et}({\cal X},\zz/n)$$
are isomorphisms for $p\le q$ and monomorphisms for $p=q+1$
\item the homomorphisms
$$\tilde{H}^{p,q}({\cal X},\zz)\sr \tilde{H}^{p,q}_{et}({\cal X},\zz)$$
are isomorphisms for $p\le q+1$ and monomorphisms for $p=q+2$
\end{enumerate}
\end{theorem}
Let $X$ be a splitting variety for a symbol $\uu{a}$. Recall that $X$
is called a generic splitting variety if for any field $E$ over $k$
such that $\uu{a}=0$ in $K_n^M(E)/l$ there exists a zero cycle on $X$
of degree prime to $l$.
\begin{theorem}
\llabel{univ} Let $l$ be a prime and $k$ be a field of chracteristic
zero. Let further $\uu{a}=(a_1,\dots,a_n)$ be a sequence of invertible
elements of $k$ and $X$ be $\nu_{n-1}$-variety which splits
$\uu{a}$. Then $X$ is a generic splitting variety for $\uu{a}$. 
\end{theorem}
\begin{proof}
It is a reformulation of Lemma \ref{isuniversal}.
\end{proof}

\def\cprime{$'$}

\comment{
\bibliography{alggeom}

\begin{thebibliography}{10}

\bibitem{kraines}
D.~Kraines.
\newblock Massey higher products.
\newblock {\em Trans.AMS}, 124:431--449, 1966.

\bibitem{Lazard}
Michel Lazard.
\newblock Lois de groupes et analyseurs.
\newblock {\em Ann. Sci. Ecole Norm. Sup. (3)}, 72:299--400, 1955.

\bibitem{MayTT}
J.~P. May.
\newblock The additivity of traces in triangulated categories.
\newblock {\em Adv. Math.}, 163(1):34--73, 2001.

\bibitem{Milnor3}
John Milnor.
\newblock The {S}teenrod algebra and its dual.
\newblock {\em Annals of Math.}, 67(1):150--171, 1958.

\bibitem{SJ}
Andrei Suslin and Seva Joukhovitski.
\newblock Norm varieties.
\newblock {\em J. Pure Appl. Algebra}, 206(1-2):245--276, 2006.

\bibitem{MCpub}
Vladimir Voevodsky.
\newblock Motivic cohomology with {${\bf Z}/2$}-coefficients.
\newblock {\em Publ. Math. Inst. Hautes \'Etudes Sci.}, (98):59--104, 2003.

\bibitem{motcoh}
Vladimir Voevodsky.
\newblock On motivic cohomology with $\zz/l$-coefficients.
\newblock {\em www.math.uiuc.edu/K-theory/639}, 2003.

\bibitem{Redpub}
Vladimir Voevodsky.
\newblock Reduced power operations in motivic cohomology.
\newblock {\em Publ. Math. Inst. Hautes \'Etudes Sci.}, (98):1--57, 2003.

\bibitem{oversub}
Vladimir Voevodsky.
\newblock Motives over simplicial schemes.
\newblock {\em arXiv:0805.4431, To appear in Journal of K-theory}, 2009.

\bibitem{SRFsub}
Vladimir Voevodsky.
\newblock Simplicial radditive functors.
\newblock {\em arXiv:0805.4434, To appear in Journal of K-theory}, 2009.

\bibitem{Red2sub}
Vladimir Voevodsky.
\newblock Motivic {E}ilenberg-{M}aclane spaces.
\newblock {\em To appear in Publ. IHES}, 2010.

\bibitem{patching}
Charles~A. Weibel.
\newblock Patching the norm residue isomorphism theorem.
\newblock {\em www.math.uiuc.edu/K-theory/844}, 2007.

\end{thebibliography}
\bibliographystyle{plain}
}

\end{document}